\documentclass[11pt,reqno]{amsart}
\usepackage[top=1.2in, bottom=1.2in, left=1.2in, right=1.2in]{geometry}
\usepackage{amsfonts, amssymb, amsmath, mathtools, dsfont, xcolor}
\usepackage[linktoc=page, colorlinks, linkcolor=blue, citecolor=blue]{hyperref}
\usepackage{enumerate, commath}

\newtheorem{theorem}{Theorem}[section]

\newtheorem{proposition}[theorem]{Proposition}
\newtheorem{lemma}[theorem]{Lemma}

\newtheorem{definition}[theorem]{Definition}

\newtheorem{remark}[theorem]{Remark}

\def\e{\varepsilon}
\def\irr#1{{\rm Irr}(#1)}
\def\irrr#1#2 {\irr {#1 \mid #2}}
\providecommand{\norm}[1]{\left\lVert#1\right\rVert}\newcommand{\R}{\mathbb R}

\newcommand{\B}{\mathcal{B}}
\newcommand{\Z}{\mathbb Z}
\newcommand{\sfe}{{{\mathbb S}^{n-1}}}

\newcommand{\E}{\mathbb E}

\newcommand{\Le}{\mathcal{L}}

\def\Prob{{\mathbb P}}
\def\Event{{\mathcal E}}

\def \HS {\mathrm{HS}}
\def \tran {\mathsf{T}}

\renewcommand{\Pr}[2][]{\mathbb{P}_{#1} \left\{ #2 \rule{0mm}{3mm}\right\}}
\newcommand{\ip}[2]{\langle#1,#2\rangle}
\DeclareMathOperator{\Span}{span}
\DeclareMathOperator{\dist}{dist}
\DeclareMathOperator{\supp}{supp}
\DeclareMathOperator{\Sparse}{Sparse}
\DeclareMathOperator{\Comp}{Comp}
\DeclareMathOperator{\Incomp}{Incomp}
\DeclareMathOperator{\RLCD}{RLCD}
\DeclareMathOperator{\Var}{Var}
\DeclareMathOperator{\Cov}{Cov}

\DeclareMathOperator{\Unif}{Unif}

\def \l {{\lambda}}
\def \EE {\mathcal{E}}

\begin{document}

\title{The smallest singular value of inhomogeneous square random matrices}

\author{Galyna V. Livshyts, Konstantin Tikhomirov \and Roman Vershynin}

\address{Georgia Institute of Technology}
\email{glivshyts6@math.gatech.edu, konstantin.tikhomirov@math.gatech.edu}

\address{University of California, Irvine}
\email{rvershyn@uci.edu}

\thanks{G.L. was partly supported by NSF grant CAREER DMS-1753260. R.V. was partly supported by U.S. Air Force grant FA9550-18-1-0031, NSF grants DMS-1954233 and DMS-2027299, and U.S. Army Grant 76649-CS}


\begin{abstract}
We show that for an $n\times n$ random matrix $A$
with independent uniformly anti-concentrated entries, such that $\E \|A\|^2_{\HS}\leq K n^2$, 
the smallest singular value $\sigma_n(A)$ of $A$ satisfies
$$
\Pr{ \sigma_n(A)\leq \frac{\varepsilon}{\sqrt{n}} } 
\leq C\varepsilon+2e^{-cn},\quad \varepsilon \ge 0.
$$
This extends earlier results \cite{RudVer-square, RebTikh}
by removing the assumption of mean zero and identical distribution of the entries across the matrix,
as well as the recent result  \cite{Liv} where the matrix was required to have i.i.d.\ rows.
Our model covers inhomogeneous matrices
allowing different variances of the entries, as long as the sum of the second moments
is of order $O(n^2)$.

In the past advances, the assumption of i.i.d.\ rows was required
due to lack of Little\-wood--Offord--type inequalities for weighted sums of non-i.i.d.\ random variables.
Here, we overcome this problem by
introducing {\it the Randomized Least Common Denominator} (RLCD) which allows to study anti-concentration properties of weighted sums of independent but not identically distributed variables. We construct efficient nets on the sphere with
{\it lattice structure}, and show that the lattice points typically have large RLCD.
This allows us to derive strong anti-concentration properties for the distance
between a fixed column of $A$ and the linear span of the remaining columns, and prove the main result.
\end{abstract}
\maketitle

\section{Introduction}

Given a random matrix $A$, the question of fundamental interest is: \emph{how likely is $A$ to be invertible}, and, more quantitatively, {\em well conditioned}?
These questions can be expressed in terms of the singular values
$\sigma_1(A)\geq\dots\geq \sigma_n(A) \ge 0$, which are defined as the square roots of the 
eigenvalues of $A^\tran A$. 
The extreme singular values are especially interesting. They can be expressed as 
\begin{equation}	\label{eq: s1 sn}
\sigma_1(A)=\max_{x\in\sfe} |Ax|
\quad \text{and} \quad 
\sigma_n(A)=\min_{x\in\sfe} |Ax|,
\end{equation}
where $\sfe$ is the unit Euclidean sphere in $\R^n$.
In this paper, we will be concerned with the smallest singular value $\sigma_n(A)$. 
Its value is nonzero if and only if $A$ is invertible, 
and the magnitude of $\sigma_n(A)$ provides us with a quantitative measure of invertibility. 

The behavior of the smallest singular values of random matrices have been extensively studied \cite{BaiYin, BVW, cook, KKS, LPRT, LitRiv,
Liv, MenPao, RebTikh, Rud-square, RudVer-square, RudVer-general, tatarko, taovu-1, taovu, taovu-annals, taovu-tall,
Tikh, Tikh-nomoments, Tikh1, TikhErd, Versh-tall}. For Gaussian random matrices with i.i.d. $N(0,1)$ entries, 
the magnitude of $\sigma_n(A)$ is of order $1/\sqrt{n}$ with high probability. This observation 
goes back to von Neumann and Goldstine \cite{Neuman}, and it was rigorously verified, with precise 
tail bounds, by Edelman \cite{edelman} and Szarek \cite{szarek}. 
Extending this result beyond the Gaussian distribution is non-trivial due to the absence of rotation invariance. After the initial progress by Tao and Vu \cite{taovu-annals} and Rudelson \cite{Rud-square}, the following lower bound on $\sigma_n(A)$ was proved by Rudelson and Vershynin \cite{RudVer-square} for matrices with sub-gaussian, mean zero, unit variance, i.i.d. entries:
\begin{equation}	\label{eq: RV}
\Prob\left\{ \sigma_n(A) \le \frac{\varepsilon}{\sqrt{n}} \right\}
\le C \varepsilon + 2 e^{-cn},
\quad \e \ge 0.
\end{equation}
This result is optimal up to positive constants $C$ and $c$ (depending only on the subgaussian moment).
It has been further extended and sharpened in various ways \cite{Liv, RebTikh, RudVer-general, tatarko, Versh-tall}.
In particular, Rebrova and Tikhomirov \cite{RebTikh} relaxed 
the sub-gaussian assumption on the distribution of the entries to just having unit variance. 

It has remained unclear, however, if one can completely drop the assumption of the identical distribution of the entries of $A$. The identical distribution seemed to be crucial in the existing versions of the Littlewood--Offord theory \cite{LitOf}, which allowed to  handle arithmetic structures that arise in the invertibility problem for random matrices. 
A partial result was obtained recently by Livshyts \cite{Liv} who proved \eqref{eq: RV} under the assumption that the rows of $A$ are identically distributed (the entries must be still independent but not necessarily i.i.d). In the present paper we remove the latter requirement as well, and thus prove \eqref{eq: RV} without any identical distribution assumptions whatsoever.

We only assume the following about the entries of $A$: (a) they are independent; (b) the sum of their second moments is $O(n^2)$, which is weaker than assuming that each entry has unit second moment; (c) their distributions are uniformly anti-concentrated, i.e. not concentrated around any single value. The latter assumption is convenient to state in terms of 
the L\'evy concentration function, which for a random variable $Z$ is defined as
$$
\Le(Z,t):=\sup_{u\in\R}\Prob\{|Z-u|<t\}, \quad t \ge 0.
$$
The following is our main result. 

\begin{theorem}[Main] \label{mainthm1}
  Let $A$ be an $n\times n$ random matrix whose entries $A_{ij}$ are independent and satisfy 
  $\sum_{i,j=1}^n \E A_{ij}^2 \leq Kn^2$ for some $K>0$ and 
  $\max_{i,j} \Le(A_{ij},1)\leq b$ for some $b\in (0,1)$. 
  Then
  $$
  \Prob\left\{\sigma_n(A)\leq \frac{\varepsilon}{\sqrt{n}}\right\}\leq C\varepsilon+2e^{-cn},
  \quad \varepsilon \ge 0.
  $$
  Here $C,c>0$ depend only on $K$ and $b$.
\end{theorem}

We would like to emphasize that prior to this paper even the problem of {\it singularity} of inhomogeneous random
matrices was not resolved in the literature. In particular, it was not known if for an $n\times n$ random matrix $B$
with independent discrete entries
(say, uniformly bounded and with variances separated from zero), the singularity probability
is {\it exponentially small} in dimension. (Theorem 1 of \cite{Liv} only implied a polynomial bound on the singularity probability, without the assumption of i.i.d. rows.)

The following theorem is the primary tool in proving the main result of the paper.

\begin{theorem}[Distances] \label{mainthm2}
  For any $K>0$ and $b\in(0,1)$ there are $r, C, c>0$ depending only on $K$ and $b$ with the following property.
  Let $A$ be a random $n\times n$ matrix as in Theorem~\ref{mainthm1}.
  Denote the columns of $A$ by $A_1,\ldots, A_n$, and define 
  $$
  H_j=\Span \left\{ A_i:\,i\neq j, \; i=1,\dots,n \right\},\quad j\leq n.
  $$  
  Take any $j\leq n$ such that $\E|A_j|^2\leq r n^2$, and let $v_j$ be a random unit vector orthogonal to $H_j$
  and measurable with respect to the $sigma$--field generated by $H_j$. Then
  $$
  \Le \left(\langle v_j, A_j\rangle,\varepsilon \right)
  \leq C\varepsilon+2e^{-cn},\quad \varepsilon \ge 0.
  $$
  In particular, for every such $j$ we have
  $$
  \Pr{\dist(A_j,H_j)\leq \varepsilon}\leq C\varepsilon+2e^{-cn},\quad \varepsilon \ge 0.
  $$
\end{theorem}

Let us 
outline how Theorem~\ref{mainthm1} can be deduced from Theorem~\ref{mainthm2}. 
The first step follows the argument in \cite{RudVer-square}, which is to
decompose the sphere into compressible and incompressible vectors.
Fix some parameters $\rho,\delta\in (0,1)$, which for simplicity can be thought of as small constants. The set of compressible vectors $\Comp(\delta,\rho)$ consists of all vectors 
on the unit sphere $\sfe$ that are within Euclidean distance $\rho$ to $\delta n$-sparse vectors
(those that have at most $\delta n$ nonzero coordinates). 
The remaining unit vectors are called incompressible, 
and we have the decomposition of the sphere:
$$
\sfe = \Comp(\delta,\rho) \cup \Incomp(\delta,\rho).
$$
By the characterization \eqref{eq: s1 sn} of the smallest singular value, 
the invertibility problem reduces to finding a uniform lower bound 
over the sets of compressible and incompressible vectors:
\begin{equation}	\label{eq: comp incomp}
\Pr{ \sigma_n(A)\leq \frac{\varepsilon}{\sqrt{n}} } 
\leq \Pr{ \inf_{x\in \Comp(\delta,\rho)}|Ax| \leq \frac{\varepsilon}{\sqrt{n}} } 
  + \Pr{ \inf_{x\in \Incomp(\delta,\rho)}|Ax| \leq \frac{\varepsilon}{\sqrt{n}} }.
\end{equation}
For the compressible vectors, Lemma 5.3 from \cite{Liv} gives the upper bound $2e^{-cn}$
on the corresponding probability in \eqref{eq: comp incomp}. 
For the incompressible vectors, we use a version of the ``invertibility via distance'' bound
from \cite{RudVer-square}, which holds for any $n \times n$ random matrix $A$ (regardless of the distribution):
\begin{equation}	\label{eq: inv via dist}
\Pr{ \inf_{x\in \Incomp(\delta, \rho)} |Ax |\leq \frac{ \varepsilon \rho}{\sqrt{n}} }
\leq \frac{4}{\delta n} \inf_{J} \sum_{j\in J} \Pr{ \dist(A_j, H_j)\leq \varepsilon },
\end{equation}
where the infimum is over all subsets $J \subset[n]$ of cardinality at least $n-\delta n/2$.
To handle the distances, we apply Theorem~\ref{mainthm2}. 
Due to our assumption $\sum_{i,j=1}^n \E A_{ij}^2 = \sum_{j=1}^n \E |A_j|^2 \leq Kn^2$, 
all except at most $K/r$ terms satisfy $\E |A_j|^2 \le rn^2$.
Denoting the set of these terms by $J$ and applying Theorem~\ref{mainthm2}, we get
$$
\Pr{ \dist(A_j, H_j)\leq \varepsilon } \le C\varepsilon+2e^{-cn}
\quad \text{for all } j \in J.
$$
Since the cardinality of $J$ is at least $n-K/r \ge n-\delta n/2$ for large $n$, 
we can substitute this bound into \eqref{eq: inv via dist} and conclude that 
the last term in \eqref{eq: comp incomp} is bounded by
$\lesssim \e + e^{-cn}$ (recall that $\delta$ is a constant and we suppress it here).
Putting all together, the probability in \eqref{eq: comp incomp} gets bounded
by $\lesssim  \e + e^{-cn}$, as claimed in Theorem~\ref{mainthm1}.

\begin{remark}
Given Theorem~\ref{mainthm1}, the second assertion of Theorem~\ref{mainthm2} can be formally strengthened as follows.
Since the matrix $A$ is shown to be singular with probability at most $2e^{-cn}$,
we have that for any $j\leq n$ and any random unit vector $v_j$ orthogonal to $H_j$,
$|\langle v_j,A_j\rangle|=\dist(A_j,H_j)$ with probability at least $1-2e^{-cn}$.
Hence, the assertion of Theorem~\ref{mainthm2} can be replaced with
$$
\Le \left(\dist(A_j,H_j),\varepsilon \right)
\leq C\varepsilon+2e^{-cn},\;\; \varepsilon \ge 0,\quad\mbox{ whenever }\quad \E|A_j|^2\leq r n^2,
$$
for some $r,c,C>0$ depending only on $K,b$.
\end{remark}

\medskip

An earlier version of Theorem~\ref{mainthm2}, under the assumption that the coordinates of $A_i$ are i.i.d., was obtained by Rudelson and Vershynin \cite{RudVer-square}. They discovered an arithmetic-combinatorial invariant of a vector (in this case, a normal vector of $H_i$), which they called an essential Least Common Denominator (LCD).
The authors of \cite{RudVer-square} proved a strong Littewood--Offord--type inequality
for linear combinations of i.i.d.\ random variables in terms of the LCD of the coefficient vector,
and thus were able to estimate $\Le \left( \dist(A_i, H_i),\varepsilon \right)$.
However, in the case when $A_i$ do not have i.i.d.\ coordinates, the essential LCD is no longer applicable.
Moreover, none of the existing Littlewood--Offord--type results
could be used even to show that the distance $\dist(A_i,H_i)$ is zero with an exponentially small probability
(which would allow to conclude that the singularity probability for the inhomogeneous random matrix is exponentially small in dimension).

In the present paper, we develop a {\em randomized} version of the least common denominator
and show how it can handle the non-i.i.d.\ coordinates.
Given a random vector $X$ in $\R^n$, 
and a (deterministic) vector $v$ in $\R^n$, as well as parameters $L>0$, $u \in (0,1)$, 
the Randomized Least Common Denominator of $v=(v_1,\dots,v_n)$ (with respect to the distribution of $X=(X_1,\dots,X_n)$) is
$$
\RLCD^X_{L,u}(v)=\inf \left\{ \theta>0:\,\E \dist^2(\theta (v_1 \bar X_1,\dots,v_n \bar X_n),\mathbb{Z}^n)<\min(u|\theta v|^2, L^2) \right\},
$$
where $\bar X_i$ denotes a symmetrization of $X_i$ defined as $\bar X_i:=X_i-X_i'$, with $X_i'$ being an independent copy of $X_i$, $i=1,2,\dots,n$
(for the sake of comparison, let us recall that the essential Least Common Denominator for random vectors with i.i.d.\ components
was defined in \cite{RudVer-general} as ${\rm LCD}(v):=\inf
\{ \theta>0:\, \dist(\theta (v_1 ,\dots,v_n ),\mathbb{Z}^n)<\min(u|\theta v|, L) \}$).
In this paper, we establish a few key properties of the RLCD,
in particular, its relation to anti-concentration as well as stability under perturbations of a vector.
Other essential elements of the proof of Theorem~\ref{mainthm2} are 
a discretization argument based on the concept of random rounding and
a double counting procedure for
estimating cardinalities of $\varepsilon$--nets.
Those were, in a rather different form, used in \cite{Liv} and \cite{TikhErd}.

In Section~\ref{s: prelims} we discuss some preliminaries and introduce our main tool, the RLCD. 
In Section~\ref{s: discretization} we outline the discretization procedure, based on the idea of random rounding.
In Section~\ref{s: double counting} we outline the key result, which informally states that ``lattice vectors are usually nice'', and is based on the idea of double counting. 
In Section~\ref{s: proof dist} we combine the results of Sections~\ref{s: discretization} and~\ref{s: double counting}, and prove Theorem~\ref{mainthm2}. 
In Section~\ref{s: proof main} we conclude by formally deriving Theorem \ref{mainthm1} from Theorem \ref{mainthm2}. 

\begin{remark}
The main results of this paper are stated here for real random matrices, and can be extended to random matrices with complex entries. This was recently done in the preprint \cite{Jain-Silwal} following the approach we presented in the present paper.
\end{remark}

\subsection*{Acknowledgement} The first author is grateful to the mathematics department of UC Irvine for hospitality. The first two authors are grateful to Mark~Rudelson for suggesting this problem. All authors thank the referees for their many useful comments.

\section{Preliminaries}		\label{s: prelims}

The inner product in $\R^n$ is denoted $\langle \cdot,\cdot\rangle$, the Euclidean norm is denoted $|\cdot|$, and the sup-norm is denoted $\|x\|_{\infty}=\max_i |x_i|$. The Euclidean unit ball and sphere in $\R^n$ are denoted $B_2^n$ and $\sfe$, respectively. The unit cube and the cross-polytope in $\R^n$ are denoted 
$$
B_{\infty}^n = \big\{ x\in\R^n:\,\|x\|_{\infty}\leq 1 \big\},
\quad
B_1^n = \big\{x\in\R^n:\,\sum_{i=1}^n |x_i|\leq 1 \big\}.
$$
The integer part of a real number $a$ (i.e., the largest integer which is smaller or equal to $a$) is denoted by $\lfloor a\rfloor$,
and the fractional part by $\{a\}=a-\lfloor a\rfloor$.
The cardinality of a finite set $I$ is denoted by $\sharp I$.

Columns of an $N\times n$ matrix $M$ will be denoted by $M_j,$ for $j=1,\dots,n,$ and the rows will be denoted $M^i,$ with $i=1,\dots,N.$

For a random variable $X$, we denote by $\overline{X}$ the symmetrization of $X$
defined as $\overline{X}=X-X'$, where $X'$ is an independent copy of $X$.
Note that 
\begin{equation}	\label{eq: symmetrization var}
\E |\overline{X}|^2 = 2\Var(X),
\end{equation}
where we defined the variance of a random vector $X$ as the covariance 
of $X$ with itself, i.e. $\Var(X) = \Cov(X,X) = \E |X - \E X|^2$.

\subsection{Decomposition of the sphere}

We shall follow the scheme developed by Rudelson and Vershynin in \cite{RudVer-square}, 
the first step of which is to decompose the sphere to the set of compressible and incompressible vectors. Such decomposition in some form goes back to earlier works, in particular that of Litvak, Pajor, Rudelson and Tomczak-Jaegermann \cite{LPRT}, and it was used in many papers since then 
\cite{RudVer-general, tatarko, Tikh, RebTikh}.

Fix some parameters $\delta,\rho\in(0,1)$ whose values will be chosen later, and define 
the sets of sparse, compressible, and incompressible vectors as follows:
\begin{gather*}
\Sparse(\delta) := \left\{ u\in\sfe: \#\supp(u) \le \delta n \right\},\\
\Comp(\delta,\rho) := \left\{ u\in\sfe:\,\dist(u,\Sparse(\delta))\leq \rho \right\}, \\
\Incomp(\delta,\rho) := \sfe\setminus \Comp(\delta,\rho).
\end{gather*}
We will use a result of \cite{Liv}, 
which gives a good uniform lower bound for $|Ax|$ 
on the set of compressible vectors:

\begin{lemma}[Lemma 5.3, \cite{Liv}]\label{comp-final}
  Let $A$ be an $N\times n$ random matrix with $N \ge n$, 
  whose entries $A_{ij}$ are independent and satisfy 
  $\sum_{i=1}^N \sum_{j=1}^n \E A_{ij}^2 \leq KNn$ for some $K>0$ and 
  $\max_{i,j} \Le(A_{ij},1)\leq b$ for some $b\in (0,1)$. 
  Then 
  $$
  \Pr{ \inf_{x\in \Comp(\delta,\rho)} |Ax|\leq c\sqrt{N} }
  \leq 2e^{-cN}.
  $$
  Here $\rho, \delta\in(0,1)$ and $c>0$ depend only on $K$ and $b$.
\end{lemma}

The rest of our argument will be about incompressible vectors. 

\subsection{Randomized Least Common Denominator}

We will need the following lemma due to Esseen (see Esseen \cite{Ess}, or, e.g., Rudelson--Vershynin \cite{RudVer-square}):

\begin{lemma}[Esseen]\label{essen}
Given a variable $\xi$ with the characteristic function $\varphi(\cdot)=\E\exp(2\pi {\bf i}\xi\cdot)$,
$$\Le(\xi,t)\leq C\int_{-1}^{1} \bigg|\varphi\left(\frac{s}{t}\right)\bigg|\,ds,\quad t>0,$$
where $C>0$ is an absolute constant.
\end{lemma}

Rudelson and Vershynin \cite{RudVer-square, RudVer-general} specialized Esseen's lemma for 
weighted sums of independent random variables 
$\ip{X}{v} = \sum_{i=1}^n v_i X_i$:

\begin{lemma}\label{concfunc}
Let $X=(X_1,\dots,X_n)$ be a random vector with independent coordinates. 
Then for every vector $v\in\R^n,$ and any $t>0$, we have\footnote{Recall that $\overline{X_i}$ denotes the symmetrization of $X_i$, which we defined in the beginning of Section~\ref{s: prelims}.}
$$\Le\left(\langle X, v\rangle,t\right)
\leq C_{\text{\tiny\ref{concfunc}}}\int_{-1}^1 \exp\bigg(-c_{\text{\tiny\ref{concfunc}}}\E\Big(\sum_{i=1}^n \Big[1-\cos\Big(\frac{2\pi s\overline{X}_i v_i}{t}\Big)\Big]
\,\Big)\bigg)ds.$$
The constants $C_{\text{\tiny\ref{concfunc}}},c_{\text{\tiny\ref{concfunc}}}>0$ are absolute. 
\end{lemma} 

For completeness, we outline the argument here. 

\begin{proof}
Let $\varphi$ be the characteristic function of $\langle X, v\rangle$, and 
$\varphi_i$ be the characteristic function of $X_i$. By independence, we have
$$\varphi(s)=\prod_{i=1}^n \varphi_i(sv_i),\quad s\in\R.$$
By definition of $\overline{X}_i$, we have for each $i\leq n$:
$$
|\varphi_i(sv_i)|=\sqrt{\E \cos(2\pi sv_i\overline{X}_i)}
\leq \exp \Big(-\frac{1}{2}\left(1-\E\cos( 2\pi sv_i \overline{X}_i)\right) \Big),\quad s\in\R,
$$
where the last step uses the inequality $|a|\leq \exp \big( -\frac{1}{2}(1-a^2) \big)$ valid for all $a\in\R$. 
To finish the proof it remains to use Lemma \ref{essen}.
\end{proof}

In analogy with the notion of the essential least common denominator (LCD) developed by Rudelson and Vershynin \cite{RudVer-square, RudVer-general, RudVer-delocalization}, we define a 
randomized version of LCD, which will be instrumental in controlling the sums non-identically 
distributed random variables.

\begin{definition}
  For a random vector $X$ in $\R^n$, 
  a (deterministic) vector $v$ in $\R^n$, and parameters $L>0$, $u \in (0,1)$, 
  define 
  $$
  \RLCD^X_{L,u}(v) 
  :=\inf \left\{ \theta>0:\,\E \dist^2(\theta v\star \overline{X},\mathbb{Z}^n)
    <\min(u|\theta v|^2, L^2) \right\}.
  $$
Here by $\star$ we denote the Schur product
$$v\star X:=(v_1 X_1,\dots,v_nX_n).$$
\end{definition}

The usefulness of RLCD is demonstrated in the following lemma, which shows how 
RLCD controls the concentration function of a sum of independent random variables.

\begin{lemma}\label{smallball-1}
  Let $X=(X_1,\dots,X_n)$ be a random vector with independent coordinates.  
  Let $c_0 > 0$, $L>0$ and $u \in (0,1)$.
  Then for any vector $v\in\R^n$ with $|v|\geq c_0$
  and any $\varepsilon \ge 0$, we have
  $$
  \Le(\langle X,v\rangle, \varepsilon)
  \leq C\varepsilon
  +C\exp(-\widetilde c L^2)
  + \frac{C}{\RLCD^X_{L,u}(v)}.
  $$
  Here $C>0,\widetilde c>0$ may only depend on $c_0,u$. 
\end{lemma}

\begin{proof}
Take any $\varepsilon\geq 1/\RLCD^X_{L,u}(v)$.
By Lemma~\ref{concfunc}, we have
$$\Le\left(\langle X, v\rangle,\varepsilon\right)
\leq C_{\text{\tiny\ref{concfunc}}}\int_{-1}^1 \exp\bigg(-c_{\text{\tiny\ref{concfunc}}}\E\Big(\sum_{i=1}^n \Big[1-\cos\Big(\frac{2\pi s\overline{X}_i v_i}{\varepsilon}\Big)\Big]
\,\Big)\bigg)ds.$$
For each $s\in[-1,1]$ and $i\leq n$ we have
$$
\E\Big[1-\cos\Big(\frac{2\pi s\overline{X}_i v_i}{\varepsilon}\Big)\Big]
\geq \widetilde c\,\E \,\dist^2(s\overline{X}_i v_i/\varepsilon,\Z)
$$
for some universal constant $\widetilde c>0$. Hence,
\begin{align*}
\Le\left(\langle X, v\rangle,\varepsilon\right)
&\leq C_{\text{\tiny\ref{concfunc}}}\int_{-1}^1 \exp\bigg(-c_{\text{\tiny\ref{concfunc}}}\widetilde c\,
\E \,\dist^2(s\overline{X} \star v/\varepsilon,\Z^n)
\,\bigg)ds\\
&=
C_{\text{\tiny\ref{concfunc}}}\varepsilon\int_{-1/\varepsilon}^{1/\varepsilon} \exp\bigg(-c_{\text{\tiny\ref{concfunc}}}\widetilde c\,
\E \,\dist^2(s\overline{X} \star v,\Z^n)
\,\bigg)ds\\
&\leq
C_{\text{\tiny\ref{concfunc}}}\varepsilon\int_{-1/\varepsilon}^{1/\varepsilon} \exp\bigg(-c_{\text{\tiny\ref{concfunc}}}\widetilde c\,
\,\min(u|s v|^2,L^2)
\,\bigg)ds,
\end{align*}
where at the last step we used the definition of RLCD and the assumption on $\varepsilon$.
A simple computation finishes the proof.
\end{proof}

We shall also need the notion of the randomized LCD for matrices.

\begin{definition}		\label{def: RLCD matrix}
For an $m\times n$ matrix $M$ with rows $M^{1},\dots, M^m$, and a vector $v\in\R^n$, define
$$
\RLCD^M_{L,u}(v):=\min_{i=1,\dots,m} \RLCD_{L,u}^{M^i}(v).$$
\end{definition}

Recall the following ``tensorization'' lemma of Rudelson and Vershynin \cite{RudVer-square}:

\begin{lemma}[Tensorization lemma, Rudelson--Vershynin \cite{RudVer-square}]\label{tensorization}
Suppose that $\varepsilon_0\in(0,1)$, $K\geq 1$, and let $Y_1,\dots,Y_m$ be independent random variables such that each $Y_i$ satisfies 
$$\Prob\{|Y_i|\leq \varepsilon\}\leq K\varepsilon\quad\mbox{for all }\varepsilon\geq\varepsilon_0.$$
Then
$$\Prob\Big\{\sum_{i=1}^m Y_i^2\leq \varepsilon^2 m \Big\}\leq (CK\varepsilon)^m,\quad \varepsilon\geq\varepsilon_0,$$
where $C>0$ is a universal constant.
\end{lemma}

The tensorization lemma is useful when one wants to control the anti-concentration 
of $|Mx|$ where $M$ is an $m \times n$ random matrix with independent rows $M^i$
and $x$ is a fixed vector. Indeed, in this case $|Mx|^2 = \sum_{i=1}^m \ip{M^i}{x}^2$,
and one can use Lemma~\ref{tensorization} for $Y_i := \ip{M^i}{x}$.
Furthermore, one can use Lemma \ref{smallball-1} 
to control the concentration function of each $Y_i$. This gives:

\begin{lemma}\label{smallball}
Let $M$ be an $m\times n$ random matrix with independent entries $M_{ij}$. 
Let $L>0$, $c_0>0$ and $u\in(0,1)$.
Then for any $x\in\R^n$ with $|x|\geq c_0$ and any 
$\varepsilon \geq C_{\ref{smallball}}\exp(-\widetilde c_{\ref{smallball}} L^2) + C_{\ref{smallball}}/\RLCD^M_{L,u}(x)$,
we have
$$
\Prob\big\{|Mx|\leq \varepsilon\sqrt{m}\big\}
\leq (C_{\ref{smallball}}\varepsilon)^m.$$
Here $C_{\ref{smallball}},\widetilde c_{\ref{smallball}}>0$ may only depend on 
$c_0$ and $u$.
\end{lemma}

A crucial property of the RLCD which will enable us to discretize the range of possible realizations
of random unit normals, is {\it stability of RLCD with respect to small perturbations}:

\begin{lemma}[Stability of RLCD]\label{stable-rlcd}
  Consider a random vector $X$ in $\R^n$ with uncorrelated coordinates, 
  a (deterministic) vector $x$ in $\R^n$, and parameters $L,u>0$. 
  Fix any tolerance level $r>0$ that satisfies
  \begin{equation}	\label{eq: r range}
  r^2 \Var(X) \le \frac{1}{8} \min \Big( u|x|^2, \, \frac{L^2}{D^2} \Big)
  \end{equation}
  where $D=\RLCD^X_{L,u}(x)$.
Then for any $y\in\R^n$ with $\|x-y\|_{\infty}<r$, we have
  $$
  \RLCD^X_{2L,4u}(y)
  \le \RLCD^X_{L,u}(x)
  \le \RLCD^X_{L/2,u/4}(y).
  $$
\end{lemma}
\begin{proof}
Note that 
\begin{equation*}
\E |x\star \overline{X} - y\star \overline{X}|^2
= \E\sum_{i=1}^n \overline{X}_i^2(x_i-y_i)^2
< r^2\E|\overline{X}|^2
= 2 r^2 \Var(X),
\end{equation*}
where the last identity is \eqref{eq: symmetrization var}.
Since $\RLCD^X_{L,u}(x)=D,$ the definition of RLCD yields
\begin{equation*}
\E \dist^2(D x\star \overline{X},\Z^n)= \min (uD^2|x|^2, L^2).
\end{equation*}
By the inequality $(a+b)^2\leq 2a^2+2b^2$, we get
\begin{equation*}
\begin{split}
\E \dist^2(D y\star \overline{X},\Z^n)
&\leq 2\E \dist^2(D x\star \overline{X},\Z^n)+2\E|Dx\star \overline{X}-Dy\star \overline{X}|^2\\
&< 2\min (uD^2|x|^2, L^2) + 4D^2 r^2 \Var(X)
\leq 4\min (uD^2|x|^2, L^2),
\end{split}
\end{equation*}
where the last step follows from our assumptions \eqref{eq: r range} on $r$. 
By definition of RLCD, this immediately gives
$$
\RLCD^X_{2L,4u}(y) \le D,
$$
which proves the first conclusion of the lemma. 

The second conclusion can be derived similarly. For any $\theta<D$, 
the definition of RLCD yields 
$$
\E\dist^2(\theta x\star \overline{X},\Z^n)
\ge \min (u\theta^2|x|^2, L^2).
$$
By the inequality $(a+b)^2\geq a^2/2-b^2$, we get
\begin{align*}
\E\dist^2(\theta y\star \overline{X},\Z^n)
&\geq \frac{1}{2} \E\dist^2(\theta x\star \overline{X},\Z^n)-\E|\theta x\star \overline{X} - \theta y\star \overline{X}|^2\\
&\geq
\frac{1}{2} \min (u \theta^2|x|^2, L^2)
- 2 \theta^2 r^2 \Var(X)
\ge \frac{1}{4} \min (u\theta^2|x|^2, L^2),
\end{align*}
where in the last step we used the bound $\theta < D$ 
and our assumptions \eqref{eq: r range} on $r$. Thus,
$$
\E\dist^2(\theta y\star \overline{X},\Z^n)\geq \min (u\theta^2|x|^2/4, L^2/4)\quad\mbox{ for all }\theta\in(0,D),
$$
and, by the definition of RLCD, this immediately gives
$$
\RLCD^X_{L/2,u/4}(y) \ge D,
$$
which proves the second conclusion of the lemma. 
\end{proof}

\medskip

The following result is a version of \cite[Lemma~3.6]{RudVer-general}. 

\begin{lemma}[Incompressible vectors have large RLCD]\label{l: aux incomp rlcd}
For any $b,\delta,\rho\in(0,1)$ there are $n_0=n_0(b,\delta,\rho)$,
$h_{\ref{l: aux incomp rlcd}}=h_{\ref{l: aux incomp rlcd}}(b,\delta,\rho)\in(0,1)$
and $u_{\ref{l: aux incomp rlcd}}=u_{\ref{l: aux incomp rlcd}}(b,\delta,\rho)\in(0,1/4)$ with the following property.
Let $n\geq n_0$, let $x\in Incomp_n(\delta,\rho)$, and assume that a random vector $X=(X_1,\dots,X_n)$ with independent components
satisfies $\Le(X_i,1)\leq b$, $i\leq n$, and $\Var |X|\leq T$, for some fixed parameter $T\geq n$. Then for any $L>0$ we have
$\RLCD^X_{L,u_{\ref{l: aux incomp rlcd}}}(x)\geq h_{\ref{l: aux incomp rlcd}}\cdot\frac{n}{\sqrt{T}}$.
\end{lemma}

\begin{proof}
For clarity of the argument, we shall often hide the parameters $b$, $\delta$, $\rho$, $h_{\ref{l: aux incomp rlcd}}$, 
and $u_{\ref{l: aux incomp rlcd}}$ in the notation such as $\lesssim, \gtrsim$;
the reader will find it easy to fill in the details. 

By definition of RLCD and since $x$ is a unit vector, it suffices to show that
$$
\E \dist^2(\theta x\star \overline{X},\mathbb{Z}^n) 
\gtrsim \theta^2
\quad \forall \; \theta \in \left(0, h_{\ref{l: aux incomp rlcd}}\cdot\frac{n}{\sqrt{T}}\right).
$$
Suppose that 
$$
\E \dist^2(\theta x\star \overline{X},\mathbb{Z}^n) \ll \theta^2
$$
for some $\theta>0$; we want to show that in this case $\theta \gtrsim \frac{n}{\sqrt{T}}$. 
Let $p \in \Z^n$ denote a closest integer vector to $\theta x\star \overline{X}$; 
note that $p$ is a random vector. 
Then $\E \abs[0]{\theta x\star \overline{X}-p}^2 \ll \theta^2$, and Markov's inequality yields that
$\abs[0]{\theta x\star \overline{X}-p} \ll \theta$
with high probability. Dividing both sides by $\theta$ gives
$\abs[0]{x \star \overline{X} - p/\theta} \ll 1$, so another application of Markov's inequality
shows that 
$$
\abs[2]{x_i \overline{X}_i - \frac{p_i}{\theta}} \ll \frac{1}{\sqrt{n}}
\quad \text{for $n-o(n)$ coordinates $i$}.
$$
Moreover, $\E \abs[1]{\overline{X}}^2 = 2 \Var \abs{X} \le 2T$ by \eqref{eq: symmetrization var}.
So a similar double application of Markov's inequality shows that, with high probability,
$$
\abs[1]{\overline{X}_i} \lesssim \sqrt{\frac{T}{n}}
\quad \text{for $n-o(n)$ coordinates $i$}.
$$

Furthermore, incompressible vectors are ``spread'' in the sense that 
$$
I \coloneqq \Big\{ i: \; \abs{x_i} \asymp \frac{1}{\sqrt{n}} \Big\} 
\quad \text{satisfies} \quad 
\abs{I} \gtrsim n.
$$
This fact is easy to check; a formal proof can be found in \cite[Lemma~3.4]{RudVer-square}.

Finally, the assumption on the concentration function shows that $\Prob \big\{\abs[1]{\overline{X}_i} \ge 1\big\} \ge b$.
By the independence of $\overline{X}_i$'s 
this implies that, with high probability,  
$$
\abs[1]{\overline{X}_i} \ge 1
\quad \text{for $b\abs{I}/2 \gtrsim n$ coordinates $i \in I$}
$$
(this conclusion follows by considering the sum of independent indicator variables ${\bf 1}_{\{|\overline{X}_i|\geq 1\}}$, $i\in I$).

Taking the intersection of these events and sets of coordinates, we see that
with high probability there must exist a coordinate $i$ for which we have simultaneously
the following three bounds:
$$
\abs[2]{x_i \overline{X}_i - \frac{p_i}{\theta}} \ll \frac{1}{\sqrt{n}}, \quad
1 \le \abs[1]{\overline{X}_i} \lesssim \sqrt{\frac{T}{n}}, \quad
\abs{x_i} \asymp \frac{1}{\sqrt{n}}.
$$
Then, using the triangle inequality, we get
$$
\abs[2]{\frac{p_i}{\theta}} 
\ge \abs[1]{x_i \overline{X}_i} - o \Big( \frac{1}{\sqrt{n}} \Big)
\ge \frac{c}{\sqrt{n}} \cdot 1 - o \Big( \frac{1}{\sqrt{n}} \Big)
> 0.
$$
Thus $p_i \ne 0$, and since $p_i$ is an integer, we necessarily have $\abs{p_i} \ge 1$. 

On the other hand, a similar application of the triangle inequality gives 
$$
\abs[2]{\frac{p_i}{\theta}} 
\le \abs[1]{x_i \overline{X}_i} + o \Big( \frac{1}{\sqrt{n}} \Big)
\lesssim \frac{1}{\sqrt{n}} \cdot \sqrt{\frac{T}{n}} + o \Big( \frac{1}{\sqrt{n}} \Big)
\lesssim \frac{\sqrt{T}}{n}.
$$
This yields that $\theta \gtrsim \abs{p_i}\cdot\frac{n}{\sqrt{T}} \ge \frac{n}{\sqrt{T}}$, as claimed. 
\end{proof}


\section{Discretization} \label{s: discretization}

In this section we outline the required discretization results. They essentially follow from the results in Section 3 of \cite{Liv}, however they are not stated there in the form we need, and thus we repeat certain arguments here. 

\begin{definition}[Discretization, part 1]
  Given a vector of weights $\alpha \in \R^n$ and a resolution parameter $\varepsilon>0$, 
  we consider the set of approximately unit vectors 
  whose coordinates are quantized at scales $\alpha_i \varepsilon/\sqrt{n}$. 
  Precisely, we define
  $$
  \Lambda_{\alpha}(\varepsilon)
  := \Big(\frac{3}{2}B_2^n\setminus\frac{1}{2}B_2^n \Big)
  \cap \left(\frac{\alpha_1 \varepsilon}{\sqrt{n}} \Z \times \dots \times \frac{\alpha_n \varepsilon}{\sqrt{n}} \Z \right).
  $$ 
\end{definition}

%
%

\begin{lemma}[Rounding]	\label{keylemmarounding}
  Fix any accuracy $\varepsilon\in (0,1/2)$, 
  a weight vector $\alpha \in [0,1]^n$, and any (deterministic) $N\times n$ matrix $A$
  whose columns we denote $A_i$.
  Then for any $x\in\sfe$ one can find $y\in \Lambda_{\alpha}(\varepsilon)$ such that
  $$
  \norm{x-y}_{\infty}\leq \frac{\varepsilon}{\sqrt{n}}
  \quad \text{and} \quad
  \abs{A(x-y)} \le \frac{\varepsilon}{\sqrt{n}} \Big( \sum_{j=1}^n \alpha^2_j \abs{A_j}^2 \Big)^{1/2}.
  $$
\end{lemma}

\begin{proof} 
Our construction of $y$ is probabilistic and amounts to \emph{random rounding} of $x$. 
The technique of random rounding has been used in computer science
(see the survey by Srinivasan \cite{Srin}, papers \cite{KA}, \cite{KV}), 
asymptotic convex geometry \cite{KlLiv} 
and random matrix theory \cite{Liv, Tikh}.

A random rounding of $x \in \sfe$ is a random vector $y$ with independent coordinates that takes values in 
the $\Lambda_{\alpha}(\varepsilon)$ and satisfies
$\E y = x$
and 
\begin{equation}	\label{eq: random rounding}
\abs{x_j-y_j} \le \frac{\alpha_j \varepsilon}{\sqrt{n}}, \quad j=1,\ldots,n,
\quad \text{for any realization of $y$}.
\end{equation}
One can construct such a distribution of $y$ by rounding each coordinate of $x$ 
up or down, at random, to a neighboring point in the lattice $(\alpha_j \e/\sqrt{n}) \Z$. 
The identity $\E y = x$ can be enforced by choosing the probabilities 
of rounding up and down accordingly.\footnote{Precisely, if 
$x_j= (\alpha_j \e/\sqrt{n})(k_j+ p_j)$ for some $k_j \in \Z$ and $p_j \in [0,1)$, 
we let $y_j$ take value $(\alpha_j \e/\sqrt{n})k_j$ with probability $1-p_j$
and value $(\alpha_j \e/\sqrt{n})(k_j + 1)$ with probability $p_j$. Clearly, 
this yields $\E y = x$.}

To check that $y$ indeed takes values in $\Lambda_{\alpha}(\varepsilon)$, 
note that the bound in \eqref{eq: random rounding} and the assumption that $\alpha_i \in [0,1]$
imply
\begin{equation}	\label{eq: norminf x-y}
\norm{x-y}_{\infty}\leq \frac{\varepsilon}{\sqrt{n}}
\quad \text{for any realization of $y$}.
\end{equation}
It follows that $\norm{x-y}_2 \le \varepsilon < 1/2$, and since $\norm{x}_2=1$, this implies
by triangle inequality that $1/2 < \norm{y}_2 < 3/2$. This verifies that the random vector
$y$ takes values in $\Lambda_{\alpha}(\varepsilon)$ as we claimed.

Finally, we have
\begin{align*} 
\E \abs{A(x-y)}^2
&= \E \abs[3]{\sum_{j=1}^n (x_j-y_j) A_j}^2
= \sum_{i=1}^n \E (x_j-y_j)^2 \cdot  \abs{A_j}^2
	\quad \text{(since $\E(x_j-y_j) = 0$)} \\
&\le \frac{\varepsilon^2}{n} \sum_{j=1}^n \alpha_j^2 \abs{A_j}^2
	\quad \text{(using the bound in \eqref{eq: random rounding})}. \\
\end{align*}
Combining this with $\eqref{eq: norminf x-y}$, we conclude that 
there exists a realization of the random vector $y$ 
that satisfies the conclusion of the lemma. 
\end{proof}

\begin{lemma}			\label{lem: ell1 net}
  Let $M \ge 1$. There exists a subset $\Xi \subset \R^n_+$ of cardinality 
  at most $(CM)^n$ and such that the following holds.
  For every vector $x \in \R^n_+$ with $\norm{x}_1 \le Mn$ there exists
  $y \in \Xi$ such that  $\norm{y}_1 \le (M+1)n$ and $y \ge x$ coordinate-wise.
\end{lemma}

\begin{proof}
Define $y \coloneqq \lceil x \rceil$ where the ceiling function is applied coordinate-wise.
Then $\norm{y}_1 \le \norm{x}_1 + n \le (M+1)n$ as claimed. 
In particular, there are as many vectors $y$ as there are integer points in the $\ell_1$-ball
$\{z \in \R^n \;:\; \norm{z}_1 \le (M+1) n\}$.
According to classical results (see \cite[Exercise 29]{PS}, \cite{Schutt}),
the number of integer points in this ball is bounded by $(CM)^n$ (see also \cite{KlLiv} for a similar covering argument).
The lemma is proved.
\end{proof}

Fix $\kappa>e$ and consider the set 
\begin{equation}\label{Omega}
\Omega_{\kappa}:=\Big\{\alpha \in [0,1]^n: \; \prod_{j=1}^n \alpha_j\geq \kappa^{-n}\Big\}.
\end{equation}

The following result is a corollary of \cite[Lemma~3.11]{Liv}.

\begin{lemma}	\label{netsonnets}
  For any $\kappa>e$ there exists a subset $\mathcal{F}\subset \Omega_{e\kappa}$ of cardinality 
  at most $({C}{\log\kappa})^n$ 
  and such that the following holds. 
  For every vector $\beta\in\Omega_{\kappa}$ there exists $\alpha\in \mathcal{F}$ 
  such that $\alpha \le \beta$ coordinate-wise.
\end{lemma} 

\begin{proof}
Apply Lemma~\ref{lem: ell1 net} for $x = -\log \beta$, $y = -\log \alpha$ (defined coordinate-wise)
and $M = \log \kappa$.
\end{proof}

\begin{definition}[Discretization -- part 2]\label{def: disc part 2} Assuming the dimension $n$ fixed, for the parameters $\kappa>e$ and $\varepsilon>0$, we shall use notation
\begin{equation}\label{mainLambda}
\Lambda^{\kappa}(\varepsilon):=\bigcup_{\alpha\in\mathcal{F}} \Lambda_{\alpha} (\varepsilon),
  \end{equation}
with $\mathcal{F}$ being the set whose existence is guaranteed by Lemma \ref{netsonnets}.
\end{definition}

\begin{remark}\label{rem: card of Lambda}
It is immediate from the above definition that for any $\kappa>e$ there is $C_\kappa>0$ depending
only on $\kappa$ such that $\sharp \Lambda^{\kappa}(\varepsilon)\leq \sum\limits_{\alpha\in\mathcal F}\sharp \Lambda_{\alpha} (\varepsilon)\leq (C_\kappa /\varepsilon)^n$
for every $\varepsilon\in(0,1]$.
\end{remark}

\medskip

The following notion from \cite{Liv} will help us to control the norms of the columns $A_j$ of 
an $N\times n$ matrix $A$ in the absence of any distributional assumptions on $A_j$:
$$
\B_{\kappa}(A)
\coloneqq \min \Big\{ \sum_{j=1}^n \alpha_j^2 |A_j|^2 \;:\; \alpha \in \Omega_\kappa \Big\}.
$$

\begin{theorem}	\label{determin}
  Fix $\varepsilon \in (0,1/2)$, $\kappa>e$, and  
  any (deterministic) $N\times n$ matrix $A$. 
  Then for every $x\in \sfe$ one can find $y\in \Lambda^{\kappa}(\varepsilon)$ so that
  $$
  \norm{x-y}_{\infty}\leq \frac{\varepsilon}{\sqrt{n}}
  \quad \text{and} \quad
  \abs{A(x-y)} \le \frac{\varepsilon}{\sqrt{n}} \sqrt{\B_{\kappa}(A)}.
  $$
\end{theorem}

\begin{proof} 
By Lemma \ref{keylemmarounding}, for any $x\in\sfe$ 
we can find $y\in \Lambda^{\kappa}(\varepsilon)$ that approximates $x$ in the $\ell_\infty$ norm 
as required, and such that
\begin{align*} 
\abs{A(x-y)}
&\le \frac{\varepsilon}{\sqrt{n}} 
	\Big( \min_{\alpha\in \mathcal{F}} \sum_{j=1}^n \alpha_j^2 \abs{A_j}^2 \Big)^{1/2} 
\le \frac{\varepsilon}{\sqrt{n}} 
	\Big( \min_{\beta\in \Omega_{\kappa}} \sum_{j=1}^n \beta_j^2 \abs{A_j}^2 \Big)^{1/2}
	\quad \text{(by Lemma~\ref{netsonnets})}\\
&= \frac{\varepsilon}{\sqrt{n}} \sqrt{\B_{\kappa}(A)}.
\end{align*}
The proof is complete.
\end{proof}


%

Lastly, we recall the important property concerning the large deviation behavior of $\B_{\kappa}$; here Lemma 3.11 from \cite{Liv} is quoted with a specific choice of parameters.

\begin{lemma}[Lemma 3.11 from \cite{Liv}]\label{ldpB}
  Let $A$ be a random matrix with independent columns.
  Then for any $\kappa>e$, we have
  $$
  \Pr{ \B_{\kappa}(A)\geq 2\E\|A\|_\HS^2 }
  \leq \left(\frac{\kappa}{\sqrt{2}}\right)^{-2n}.
  $$
\end{lemma}

\medskip
Finally, we are ready to state the main result of this section, which will follow as a corollary of Lemma \ref{stable-rlcd}, Theorem \ref{determin} and Lemma \ref{ldpB}. Given $\gamma>0, \omega\in (0,1), D>0$, and a distribution of a random matrix $M,$ we shall use notation 
\begin{equation*}\label{levelsets}
\begin{split}
S^M_{\omega,\gamma}(D)&:=\left\{x\in\frac{3}{2}B_2^n\setminus\frac{1}{2}B_2^n:\, \RLCD^M_{\gamma\sqrt{n},\omega}(x)\in [D,2D]\right\},\\
\tilde{S}^M_{\omega,\gamma}(D)&:=\left\{x\in\frac{3}{2}B_2^n\setminus\frac{1}{2}B_2^n:\,\RLCD^M_{2\gamma\sqrt{n},4\omega}(x)\leq 2D,\,\, \RLCD^M_{0.5\gamma\sqrt{n},0.25\omega}(x)\geq D\right\}
\end{split}
\end{equation*}
for the level sets of the RLCD.

\begin{theorem}[Approximation]\label{main-nets}
Fix any $\varepsilon\in(0,0.1)$, $\kappa>e,$ $\gamma>0,$ $\omega\in (0,1),$ $K>0$. Let $M$ be an $m\times n$ random matrix with independent columns, and whose rows $M^i$ satisfy
\begin{equation}\label{cond}
\varepsilon^2 \Var(M^i) \le \frac{1}{8} \min \Big( \omega n, \, \frac{\gamma^2n^2}{D^2} \Big), 
\quad i=1,\ldots,m.
\end{equation}
Then, with probability at least $1-(\kappa/\sqrt{2})^{-2n}$, for every $x\in\sfe\cap S^M_{\omega,\gamma}(D)$ there exists $y\in \Lambda^{\kappa}(\varepsilon) \cap \tilde{S}^M_{\omega,\gamma}(D)$ such that 
\begin{equation}\label{concl}
\|x-y\|_\infty\leq \frac{\varepsilon}{\sqrt{n}},\quad \abs{M(x-y)} \leq \frac{\sqrt{2}\varepsilon}{\sqrt{n}} \Big( \E \norm{M}^2_\HS \Big)^{1/2}.
\end{equation}
\end{theorem}

\begin{proof}
Lemma \ref{ldpB} says that the event 
$$
\mathcal{E}:=\{\B_{\kappa}(M)\leq 2\E\norm{M}^2_\HS\}
$$
occurs with probability at least $1-(\kappa/\sqrt{2})^{-2n}$.
Fix any realization of the random matrix $M$ for which this event happens.

Let $y$ be the approximation of $x$ given by Theorem \ref{determin}. Then (\ref{concl}) follows from the conclusion of Theorem \ref{determin} and the definition of our event. The fact that $y\in \tilde{S}^M_{\omega,L}(D)$ follows from Lemma \ref{stable-rlcd} (applied with $r=\epsilon/\sqrt{n}$) together with the assertion of Theorem \ref{determin} (applied with $A=M$): indeed, the assumption (\ref{cond}) allows us to appeal to Lemma \ref{stable-rlcd}. 
\end{proof}

%
%

\section{Anti-concentration on lattice points}		\label{s: double counting}


The goal of this section is to study anti-concentration properties of random sums with
coefficients taken from sets of the form
\begin{equation}	\label{eq: Lambda}
\Lambda:=\left(\frac{3}{2}B_{2}^n\cap\big\{x\in\R^n:\;
\sharp\{i:\;|x_i|\geq \frac{\rho}{\sqrt{n}}\}\geq \delta n\big\}\right)
\cap \left( \frac{\l_1}{\sqrt{n}} \Z \times \cdots \times  \frac{\l_n}{\sqrt{n}} \Z \right).
\end{equation}

The main result of this section is the following

\begin{theorem}[Most lattice points are unstructured]				\label{mainprop}
For any $U\geq 1$, $b\in(0,1)$ and $\delta,\rho\in(0,1/2]$ there exist $n_0=n_0(U,b,\delta,\rho)$,
 $\gamma=\gamma(U,b,\delta,\rho)\in(0,1)$
and $u=u(b,\delta,\rho)\in(0,1/4)$ such that the following holds.
Let $n\geq n_0$.
Consider a random vector $X$ in $\R^n$ with independent components $X_i$ that satisfies
$$
\Var(X)\leq \frac{1}{8}(1-b)\delta \gamma^2 n^2
\quad \text{and} \quad 
\max_i \Le(X_i,1)\leq b.
$$
Fix numbers $\l_1,\ldots,\l_n$ satisfying $6^{-n} \le \l_i \le 0.01$ 
and let $W$ be a vector uniformly distributed on the set $\Lambda$ defined in \eqref{eq: Lambda}.
Then
$$
\Prob_W \Big\{ \RLCD^X_{\gamma\sqrt{n},u}(W)< \min_i 1/\l_i \Big\} \leq U^{-n}.
$$
\end{theorem}

The above theorem will be used to control the cardinality of $\varepsilon$-nets
on the set of ``typical'' realizations of unit normal vectors to the spans of columns of our random matrix,
and forms a crucial step in the proof of Theorem~\ref{mainthm2}.
The idea of using double counting to verify structural properties of random normals
was applied earlier in \cite{TikhErd}.

We start with an observation that will allow us to reduce the Euclidean ball $\frac{3}{2}B_{2}^n$
by a parallelotope in the definition of $\Lambda$.

\begin{lemma}\label{l: ball to rect}
There is a universal constant $C_0>0$ with the following property.
For any $n\geq 1$, there is a collection of parallelotopes $\mathcal P=\{P_i\}$ in $\R^n$
of cardinality at most $2^{C_0n}$, such that
\begin{itemize}
\item Each $P_i$ is centered at the origin, with the edges parallel to the coordinate axes;
\item Each edge of $P_i$ is of length at least $2/\sqrt{n}$;
\item $\frac{3}{2}B_2^n\subset\bigcup\limits_i P_i\subset 3B_2^n$.
\end{itemize}
\end{lemma}
\begin{proof}
First, standard volumetric estimates imply that there is a covering of $\frac{3}{2}B_2^n$ by parallel translates of the cube $\frac{1}{2\sqrt{n}}B_\infty^n$,
of cardinality at most $2^{C_0n}$ for a universal constant $C_0>0$.
Let $\{x_i\}_{i\in I}$ be a collection of at most $2^{C_0n}$ points in $\frac{3}{2}B_2^n$ such that each of the cubes from the covering
contains at least one point $x_i$ from the collection.
Now, define $\mathcal P=\{P_i\}_{i\in I}$ by taking, for each $i\in I$, $P_i:=\widetilde P_i+\frac{1}{\sqrt{n}}B_\infty^n$,
where $\widetilde P_i$ is the unique parallelotope centered at the origin, and with $x_i$ being one of its vertices.
It is elementary to check that the collection satisfies the required properties.
\end{proof}

\begin{lemma}\label{l: aux1}
For any $b\in(0,1)$ and $\delta,\rho\in(0,1/2]$, there exists $n_0=n_0(b,\delta,\rho)$
such that the following holds.
Let $n\geq n_0$ and $\gamma \in (0,1)$. 
Fix any subset $J \subset [n]$
and consider a fixed (deterministic) vector $x \in \R^n$ satisfying
\begin{equation}	\label{eq: x conditions}
  |x|^2\leq \frac{1}{4}(1-b)\delta \gamma^2 n^2
  \quad \text{and} \quad 
  \sharp\{i\in J:\;|x_i|\geq 1\} \geq \frac{1}{2}(1-b)\delta n.
\end{equation}
Furthermore, fix numbers $\l_1,\ldots,\l_n$ satisfying $6^{-n} \le \l_i \le 0.01$ and
a vector $a = (a_1,\dots,a_n)$ satisfying 
$\abs{a} \le 3$ and $\min a_i \ge 1/\sqrt{n}$. 
Consider the parallelotope $P:=\prod_{i=1}^n [-a_i,a_i]$, 
and define
$$
\Lambda':=\left\{ w \in P:\; |w_i| \geq \frac{\rho}{\sqrt{n}} \; \forall i\in J \right\}
\cap \left( \frac{\l_1}{\sqrt{n}} \Z \times \cdots \times  \frac{\l_n}{\sqrt{n}} \Z \right).
$$
Let $W$ be a random vector uniformly distributed on $\Lambda'$. 
Then, for $D \coloneqq \min_i 1/\l_i$, we have
\begin{equation}\label{eq: outcome}
\Prob\Big\{ \min_{\theta \in (0,D)} \dist(\theta W\star x,\mathbb{Z}^n)^2 < \min \big( c|\theta W|^2/2,16\gamma^2n \big) \Big\}
\leq (C\gamma)^{cn},
\end{equation}
where $C,c>0$ depending only on $b,\delta,\rho$.
\end{lemma}

\begin{proof}
{\bf Step 1. Halving the set $I$.}
The assumptions on $x$ imply that the set
$$
I:=\big\{i\in J:\;1\leq |x_i|\leq \gamma\sqrt{n}\big\}
\quad \text{satisfies} \quad
\sharp I \ge \frac{1}{4} (1-b)\delta n.
$$
Next, let $\mu=\mu(x)$ be a median of the set $\{ a_i\abs{x_i} :\; i \in I \}$.
Thus, each of the subsets 
$$
I' \coloneqq \{i\in I :\; \;a_i\abs{x_i}\leq \mu\}
\quad \text{and} \quad
I'' \coloneqq \{i\in I :\; a_i\abs{x_i}\geq \mu\}
$$
contains at least a half of the elements of $I$:
\begin{equation}	\label{eq: halves sizes}
\min( \sharp I', \sharp I'') 
\ge \frac{1}{2} \sharp I
\ge \frac{1}{8} (1-b)\delta n
\ge cn,
\end{equation}
where $c>0$ depends only on $b$ and $\delta$.
Take $\theta\in(0,D)$ and consider two cases.

\medskip
{\bf Step 2. Ruling out small multipliers $\theta$.}
We claim that the range for $\theta$ in \eqref{eq: outcome} can automatically 
be narrowed to $(\frac{1}{2\mu}, D)$. 
To check this, it suffices to show that for any $\theta \in (0, \frac{1}{2\mu}]$,
the bound
\begin{equation}	\label{eq: small multipliers}
\dist(\theta W\star x,\mathbb{Z}^n)^2 \ge c |\theta W|^2/2
\end{equation}
holds deterministically, i.e. for any realization of the random vector $W$.

By construction, the coordinates $W_i$ of $W$ for $i \in I$ are uniformly distributed 
in lattice intervals, namely
\begin{equation}	\label{eq: Wi range}
W_i \sim \Unif \Big( \Big[ \frac{\rho}{\sqrt{n}}, a_i \Big] \cap \frac{\l_i}{\sqrt{n}} \Z \Big), 
\quad i \in I.
\end{equation}
This means in particular that the coordinates of $\theta W\star x$ for $i \in I'$ satisfy
$$
\theta \abs{W_i x_i} 
\le \theta a_i \abs{x_i}
\le \theta \mu \le \frac{1}{2},
$$
where we used the definition of $I'$ and the smallness of $\theta$.
This bound in turn yields
$$
\dist(\theta \abs{W_i x_i}, \Z) 
= \theta \abs{W_i x_i}
\ge \theta \cdot \frac{\rho}{\sqrt{n}} \cdot 1
$$
where in the last step we used the range of $W_i$ from \eqref{eq: Wi range} 
and the definition of $I$.
Square both sides of this bound and sum over $i \in I'$ to get
$$
\dist(\theta W\star x, \Z^n)^2
\ge \frac{\theta^2 \rho^2}{n} \sharp I'
\ge c \theta^2 \rho^2
\ge c_0 \theta^2 \abs{W}^2/2,
$$
where we used \eqref{eq: halves sizes}, 
suppressed $\rho$ into $c_0$, and 
noted that $\abs{W}^2 \le \abs{a}^2 \le 9$ 
by definition of $W$ and assumption on $a$. 
We have proved \eqref{eq: small multipliers}.

\medskip
{\bf Step 3. Handling a fixed multiplier $\theta$.}
Due to the previous step, our remaining task is to show that 
$$
\Prob\Big\{ \min_{\theta \in (1/2\mu,D)} \dist(\theta W\star x,\mathbb{Z}^n)^2 <
16\gamma^2n \Big\}
\leq (C\gamma)^{cn}.
$$
To do this, let us first estimate the probability 
that $\dist(\theta W\star x,\mathbb{Z}^n)^2 <49\gamma^2n$
for a {\em fixed} multiplier\footnote{Extending the range by $1$ will be help us in the next step to unfix $\theta$;
increasing the constant factor $16$ to $49$ will help us run a net approximation argument in Step 4.}  $\theta \in (1/2\mu,D+1)$.

Let $i \in I''$. Recall from \eqref{eq: Wi range} that the random variable $\abs{W_i}$ 
is uniformly distributed in a lattice interval whose diameter is at least
$$
a_i - \frac{\rho}{\sqrt{n}} - \frac{2\l_i}{\sqrt{n}}
\ge \frac{a_i}{3};
$$
here we used the assumptions $a_i \ge 1/\sqrt{n}$, $\rho \le 1/2$ and $\l_i \le 0.01$.
Thus, the random variable $\theta \abs{W_i x_i}$,
i.e. the absolute value of a coordinate of $\theta W\star x$, 
is distributed in a lattice interval of diameter at least 
$$
\frac{a_i}{3} \theta \abs{x_i}
\ge \frac{\theta \mu}{3} 
\ge \frac{1}{6};
$$
here we used the definition of $I''$ and the largeness of $\theta$.
Moreover, the step of that lattice interval (the distance between any adjacent points) is
$$
\frac{\l_i}{\sqrt{n}} \theta \abs{x_i}
\le \l_i \theta \gamma
\le \l_i (D+1) \gamma
\le 2\gamma;
$$
here we used the definition of $I$, the range of $\theta$, the definition of $D$,
and the assumption that $\l_i \le 0.01$.

The random variable $\theta \abs{W_i x_i}$ 
that is uniformly distributed on a lattice interval of diameter at least $1/6$
and with step at most $2\gamma$ satisfies
$$
\Pr{\dist(\theta \abs{W_i x_i}, \Z) < \e} \le C\e
\quad \text{for any } \e \ge 4 \gamma,
$$
where $C$ is an absolute constant.
Squaring the distances, summing them over $i \in I''$ 
and using Tensorization Lemma~\ref{tensorization},
we conclude that 
$$
\Pr{ \dist(\theta W\star x,\mathbb{Z}^n)^2 < \e^2 \sharp I'' }
\le (C'\e)^{\sharp I''}
\quad \text{for any } \e \ge 4 \gamma.
$$
Recall from \eqref{eq: halves sizes} that $\sharp I'' \ge cn$. 
Hence, substituting $\e = C_0 \gamma$ with sufficiently large $C_0$
(depending on $c$ and thus ultimately on $b$ and $\delta$), we get 
$$
\Pr{ \dist(\theta W\star x,\mathbb{Z}^n)^2 < 49 \gamma^2 n }
\le (C''\gamma)^{cn}.
$$

\medskip
{\bf Step 4. Unfixing the multiplier $\theta$.}
It remains to make the distance bound hold simultaneously for all $\theta$ in the 
range $(1/2\mu,D)$. To this end, we use a union bound combined with a
discretization argument. To discretize the range of $\theta$, consider the 
lattice interval
$$
\Theta \coloneqq \Big( \frac{1}{2\mu},D \Big) \cap \frac{1}{\sqrt{n}} \Z.
$$
For sufficiently large $n$, its cardinality can be bounded as follows:
$$
\sharp \Theta \le (D+1) \sqrt{n}+1
\le (6^n + 1)\sqrt{n} + 1 \le 7^n;
$$
here we used that $D = \min_i(1/\l_i)$ by definition, and $\l_i \ge 6^{-n}$ by assumption.
The construction of $\Theta$ shows that any $\theta \in (1/2\mu,D)$ can be approximated
by some $\theta_0 \in \Theta$ in the sense that
$$
\theta \le \theta_0 \le \theta + \frac{1}{\sqrt{n}}.
$$
Note in particular that $\theta_0$ falls in the range $(1/2\mu, D+1)$, which 
we handled in the previous step of the proof. 

Recall that we need to bound the probability of the event 
$$
\EE \coloneqq 
\Big\{ \min_{\theta \in (1/2\mu,D)} \dist(\theta W\star x,\mathbb{Z}^n) < 4\gamma \sqrt{n} \Big\}.
$$
Suppose this event occurs. Let $\theta$ be the multiplier that realizes the minimum
and consider an approximation $\theta_0 \in \Theta$ as above. 
By triangle inequality, it satisfies
$$
\dist(\theta_0 W\star x,\mathbb{Z}^n) 
< 4\gamma \sqrt{n} + \abs{\theta_0 - \theta} \abs{W \star x}.
$$
By construction, we have $\abs{\theta_0 - \theta} \le 1/\sqrt{n}$ 
and 
$$
\abs{W \star x} 
\le \norm{W}_\infty \abs{x}
\le 3 \gamma n;
$$
here we used that $\norm{W}_\infty \le \norm{a}_\infty \le |a| \le 3$ by definition 
of $W$ and assumptions on $a$,
as well as $\abs{x} \le \gamma n$ by assumption on $x$.
Thus,
$$
\dist(\theta_0 W\star x,\mathbb{Z}^n) 
\le 7 \gamma n.
$$ 
For each fixed $\theta_0$, the result of the previous step of the proof 
shows that the probability of this event is at most $(C''\gamma)^{cn}$.

As we know, the number of possible choices of $\theta$ is at most $\sharp \Theta \le 7^n$. 
Thus, the union bound gives
$$
\Prob(\EE) \le 7^n (C''\gamma)^{cn} \le (C\gamma)^{cn}.
$$
This completes the proof of the lemma.
\end{proof}

\begin{remark}
Note that with our choice of parameters, $\Lambda'$ is non-empty, and therefore $W$ is well-defined in the Lemma above.
\end{remark}

From Lemma \ref{l: aux1} we deduce

\begin{lemma}\label{l: aux 03982}
For any $U\geq 1$, $b\in(0,1)$ and $\delta,\rho\in(0,1/2]$, 
there exist $n_0=n_0(U,b,\delta,\rho)$,
 $\gamma=\gamma(U,b,\delta,\rho)\in(0,1)$
and $u=u(b,\delta,\rho)\in(0,1/4)$ such that the following holds.
Let $n\geq n_0$, and let $J$ be a fixed subset of $[n]$ of cardinality at least $\delta n$.
Further, consider a random vector $X$ in $\R^n$ 
with independent components $X_i$ that satisfies 
$$
\E|X|^2\leq \frac{1}{8}(1-b)\delta \gamma^2 n^2
\quad \text{and} \quad
\max_i \Le(X_i,1) \leq b.
$$
Consider a set $\Lambda'$ described in Lemma~\ref{l: aux1}
and a random vector $W$ uniformly distributed on $\Lambda'$.
Then
$$
\Prob_W\big\{\RLCD^X_{\gamma\sqrt{n},u}(W)< \min_i 1/\l_i \big\}\leq U^{-n}.
$$
\end{lemma}
\begin{proof}
We apply a simple argument based on change of integration order, or a ``double-coun\-ting'' trick.
Without any loss of generality, we can assume that the random vector $X$ is uniformly distributed
on a finite set $\mathcal X:=\mathcal X_1\times\dots\times \mathcal X_n$, so that for any $x\in \mathcal X$, we have
$$
\Prob\{X=x\}=\frac{1}{\sharp \mathcal X}.
$$
Indeed, this follows from a simple fact that any multidimensional distribution $\zeta=(\zeta_1,\dots,\zeta_n)$
with independent components 
can be approximated by a discrete distribution $\tau=(\tau_1,\dots,\tau_n)$
of the above form, so that
$$
\sup\limits_{\theta\in[0,6^n]}
\sup\limits_{v\in S^{n-1}}\big|\E \dist^2(\theta (v_1 \bar \zeta_1,\dots,v_n \bar \zeta_n),\mathbb{Z}^n)-
\E \dist^2(\theta (v_1 \bar \tau_1,\dots,v_n \bar \tau_n),\mathbb{Z}^n)
\big|
$$
is arbitrarily small. Then the definition of RLCD would imply that proving the required assertion for $\tau$
implies corresponding assertion for $\zeta$, perhaps with a different choice of $\gamma,u,n_0$.

Set $\mathcal X':=\{x\in\mathcal X:\;\mbox{$x$ satisfies \eqref{eq: x conditions}}\}$.
In view of our assumptions on $X$ (and assuming that $n$ is sufficiently large), we have
$$
\Prob\{X\in\mathcal X'\}\geq 1/4,
$$
while, in view of the assertion of Lemma~\ref{l: aux1} and summing over $x \in \mathcal{X}'$, we get
\begin{align}\label{eq: 3948724098217}
\sharp\big\{(x,w)\in\mathcal X'\times\Lambda':
&\min_{\theta \in (0,D)} \dist(\theta w\star x,\mathbb{Z}^n)^2 \ge \min(c|\theta w|^2/2,16\gamma^2n)\big\}\\
&\geq \big(1-(C\gamma)^{cn}\big)\,\sharp\mathcal X'\,\sharp\Lambda',\nonumber
\end{align}
where $D = \min_i 1/\l_i$.
This implies
\begin{multline*}
\sharp\big\{w\in \Lambda':
\sharp\{x\in\mathcal X': \min_{\theta \in (0,D)} \dist(\theta w\star x,\mathbb{Z}^n)^2 \ge \min(c|\theta w|^2/2,16\gamma^2n)\}
\geq \sharp\mathcal X'/4\big\}\\
\geq \big(1-2(C\gamma)^{cn}\big)\,\sharp\Lambda'
\end{multline*}
(indeed, if the last assertion were not true, we would get that the cardinality of the set in \eqref{eq: 3948724098217}
was bounded above by $(1-2(C\gamma)^{cn})\,\sharp\Lambda'\,\cdot\,\sharp\mathcal X'
+2(C\gamma)^{cn}\,\sharp\Lambda'\,\cdot\,\sharp\mathcal X'/4\leq (1-3(C\gamma)^{cn}/2)\sharp\mathcal X'\,\sharp\Lambda'$).
Back from counting to probabilities, we get from the last bound and the estimate $\sharp\mathcal X'/4\geq \sharp X/16$:
$$
\sharp\big\{w\in \Lambda':\, \min_{\theta \in (0,D)} \E_X\,\dist(\theta w\star X,\mathbb{Z}^n)^2
 \ge \min(c|\theta w|^2/32,\gamma^2n) \big\}
\geq
\big(1-2(C\gamma)^{cn}\big)\,\sharp\Lambda'.
$$
This can be equivalently rewritten with $u \coloneqq c/32$ as
$$
\sharp\big\{w\in \Lambda':\;\RLCD^X_{\gamma\sqrt{n},u}(w) > D\big\}\geq \big(1-2(C\gamma)^{cn}\big)\,\sharp\Lambda',
$$
and the result follows by taking any $\gamma \in (0,1)$ satisfying $2(C\gamma)^{cn}\leq U^{-n}$.
\end{proof}

\medskip

\begin{proof}[Proof of Theorem~\ref{mainprop}]
Without loss of generality, $\E X=0$, so that $\Var(X)=\E|X|^2$.
We obtain the results as a combination of Lemmas~\ref{l: ball to rect} and \ref{l: aux 03982}.
To do so, note that $\Lambda$ can be covered by $2^{C_1n}$ sets of the type $\Lambda'$ 
(one for each paralellotope and a support set $J$). Then the probability measures on 
$\Lambda$ and a given $\Lambda'$ are within $2^{C_1n}$ from each other. Thus the probability in the conclusion of Theorem~\ref{mainprop} is bounded by $2^{C_1n} U^{-n} \le (cU)^{-n}$. 
It remains to re-define $U\to cU$ to get the result.
\end{proof}

\section{Proof of Theorem \ref{mainthm2}}		\label{s: proof dist}



In this section, we split the Euclidean unit sphere $S^{n-1}$ into
{\it level sets} collecting (incompressible) unit vectors having comparable RLCD.
To show that with a high probability the normal vector does not belong to a level set with a small RLCD,
we consider a discrete approximating set whose cardinality is well controlled from above,
by using a combination of Theorem \ref{main-nets} and Theorem~\ref{mainprop}.
In view of the stability property of RLCD, the event that the normal vector has a small RLCD
is contained within the event that one of the vectors in the approximating set has a small RLCD.
We then apply the small ball probability estimates for individual vectors, combined with the union bound,
to show that the latter event has probability close to zero.

\medskip

For any $D\geq 1$, $\gamma,u\in(0,1)$, and an $m\times n$ random matrix $M$, define, as before,
$$S_D(M,\gamma,u):=\{v\in \sfe:\, \RLCD^M_{\gamma\sqrt{n},u}\in [D,2D]\}.$$

As the first step, we combine the approximation Theorem~\ref{main-nets} with Theorem~\ref{mainprop} to obtain

\begin{proposition}\label{p: discrete complete}
For arbitrary $b,\rho,\delta\in(0,1)$, $U\geq 1$ and $K\geq 1$ there exist
$n_{\ref{p: discrete complete}}=n_{\ref{p: discrete complete}}(b,\delta,\rho,U,K)$,
$u_{\ref{p: discrete complete}}=u_{\ref{p: discrete complete}}(b,\delta,\rho)
\in(0,u_{\ref{l: aux incomp rlcd}}(b,\delta,\rho))$, $\gamma_{\ref{p: discrete complete}}
=\gamma_{\ref{p: discrete complete}}(b,\delta,\rho,U,K)\in(0,1/2)$
 with the following property.
Let $D\geq 1$ and $0<\varepsilon\leq 1/D$.
Let $n\geq n_{\ref{p: discrete complete}}$, 
$m\geq 1$,
and let $M$ be an $m\times n$ matrix with independent entries $M_{ij}$ such that
$\Le(M_{ij},1)\leq b$ for all $i,j$;
$$
\Var(M^\top e_i)\leq \frac{1}{8}\min\Big((1-b)\delta\gamma_{\ref{p: discrete complete}}^2 n^2,
\varepsilon^{-2} u_{\ref{p: discrete complete}}n\Big)
$$
for every $i\leq m$, and
$$
\E\|M\|_\HS^2\leq K n^2.
$$ 
Then there is a non-random set $\Lambda\subset\R^n$ of cardinality at most $(\varepsilon U)^{-n}$
having the following properties:
\begin{itemize}
\item For any $y\in \Lambda$, we have $3/2\geq |y|\geq 1/2$;
\item For any $y\in \Lambda$, $\RLCD^M_{\gamma_{\ref{p: discrete complete}} \sqrt{n}/2,u_{\ref{p: discrete complete}}/4}(y)\geq D$ and
$\RLCD^M_{2\gamma_{\ref{p: discrete complete}} \sqrt{n},4u_{\ref{p: discrete complete}}}(y)\leq 2D$;
\item With probability at least $1-e^{-n}$, for any $x\in S_D(M,\gamma_{\ref{p: discrete complete}},
u_{\ref{p: discrete complete}})\cap \Incomp(\delta,\rho)$ there is
$y\in \Lambda$ with $\|x-y\|_\infty\leq \varepsilon/\sqrt{n}$ and $|M (x-y)|\leq \varepsilon\sqrt{n}$.
\end{itemize}
\end{proposition}
\begin{proof}
Set $\kappa:=5$, and let $C_\kappa>0$ be the constant from Remark~\ref{rem: card of Lambda}.
Let $U\geq 1$, $U':=100\sqrt{2K}U C_\kappa/\rho$, and set
$$n_{\ref{p: discrete complete}}:=n_0(U',b,\delta,\rho/2),\;
\gamma=\gamma_{\ref{p: discrete complete}}:=\gamma(U',b,\delta,\rho/2),\;
u=u_{\ref{p: discrete complete}}:=u(b,\delta,\rho/2)\in(0,\frac{1}{4}),$$
where the functions $n_0(\cdot),\gamma(\cdot),u(\cdot)$
are taken from Theorem~\ref{mainprop}.
Finally, set
$$\varepsilon':=\frac{\rho\varepsilon}{100\sqrt{2\max(K,1)}}\in(0,0.01),$$
and let $\Lambda^\kappa(\varepsilon')$ be as in Definition~\ref{def: disc part 2}.

Let $\Lambda$ be a subset of all vectors $y\in \Lambda^\kappa(\varepsilon')$
such that
$$\RLCD^M_{\gamma \sqrt{n}/2,u/4}(y)\geq D\quad\mbox{ and }\quad
\RLCD^M_{2\gamma \sqrt{n},4u}(y)\leq 2D,$$
and, such that the $\ell_\infty$--distance
of $y$ to $\Incomp(\delta,\rho)$ is at most $\varepsilon'/\sqrt{n}$. Note that the last condition implies that for any $y\in\Lambda$,
$\sharp\{i\leq n:\;|y_i|\geq \rho/(2\sqrt{n})\}\geq \delta n$, see the argument in Lemma 3.4 from \cite{RudVer-square}.

By our choice of $\varepsilon'$ and the condition on the matrix, we have
$$
(\varepsilon')^2 \Var(M^\top e_i)\leq \frac{1}{8}\frac{\gamma^2 n^2}{D^2};\quad
(\varepsilon')^2 \Var(M^\top e_i)\leq \frac{1}{8} un.
$$
Then, according to Theorem~\ref{main-nets}, with probability at least $1-(5/\sqrt{2})^{-2n}$ for any incompressible vector $x\in S_D(M,\gamma,u)$
there is a vector $y\in \Lambda$ such that $\|x-y\|_\infty \leq \varepsilon'/\sqrt{n}$ and
$|M(x-y)|\leq \sqrt{2}\varepsilon'\sqrt{K}\sqrt{n}\leq \varepsilon\sqrt{n}$.

It remains to estimate the cardinality of $\Lambda$.
We recall that
$$
\Lambda^{\kappa}(\varepsilon')=\bigcup_{\alpha\in\mathcal{F}} \Lambda_{\alpha} (\varepsilon'),
$$
where the collection $\mathcal F$ of parameters $(\alpha_1,\dots,\alpha_n)\in(0,1]^n$ is given by Lemma~\ref{netsonnets}.
Fix for a moment any $(\alpha_1,\dots,\alpha_n)\in\mathcal F$, and set
$\lambda_i:=\alpha_i\varepsilon'\in(0,0.01]$, $i\leq n$.
Observe that $1/\lambda_i\geq 1/\varepsilon'> 2/\varepsilon\geq 2D$, $i\leq n$.
Hence, we can apply Theorem~\ref{mainprop} to obtain
$$
\sharp(\Lambda\cap \Lambda_{\alpha} (\varepsilon'))\leq \sharp\Lambda_{\alpha} (\varepsilon')\,(U')^{-n}.
$$
Taking the union over all $(\alpha_1,\dots,\alpha_n)\in\mathcal F$, we then get
$$
\sharp\Lambda\leq (U')^{-n}\sum\limits_{\alpha\in\mathcal F}\sharp\Lambda_{\alpha} (\varepsilon')
\leq (\varepsilon U)^{-n},
$$
where at the last step we used our definition of $U'$.
\end{proof}

Next, we combine the discrete approximation set introduced above, with the small ball probability of Lemma~\ref{smallball}:

\begin{proposition}\label{pevelsets}
For any $b,\rho,\delta\in(0,1)$ and $K\geq 1$ there are $n_{\ref{pevelsets}}=n_{\ref{pevelsets}}(b,\delta,\rho,K)$,
$u_{\ref{pevelsets}}=u_{\ref{pevelsets}}(b,\delta,\rho)\in(0,u_{\ref{l: aux incomp rlcd}}(b,\delta,\rho))$,
$\gamma_{\ref{pevelsets}}=\gamma_{\ref{pevelsets}}(b,\delta,\rho,K)\in(0,1/2)$
and $\gamma'_{\ref{pevelsets}}=\gamma'_{\ref{pevelsets}}(b,\delta,\rho,K)$ with the following property.
Let $n\geq n_{\ref{pevelsets}}$, $e^2\leq D\leq D_0\leq e^{\gamma'_{\ref{pevelsets}} n}$,
$0\leq k\leq n/\ln D_0$, $m:=n-k$, 
and let $M$ be an $m\times n$ random matrix with independent entries $M_{ij}$ such that
$\Le(M_{ij},1)\leq b$ for all $i,j$;
\begin{equation}\label{eq: aux 498750983275}
\Var(M^i)\leq \frac{1}{64}\min\Big((1-b)\delta\gamma_{\ref{pevelsets}}^2 n^2,D_0^2 u_{\ref{pevelsets}}n\Big)
\end{equation}
for every $i \le m$, 
and
$$
\E\|M\|_\HS^2\leq K n^2.
$$ 
Let $M^{(1)}$ be the matrix obtained from $M$ by removing the first row. Then
$$
\Prob\big\{\mbox{$\exists$ $x\in \Incomp(\delta,\rho)\cap S_D(M,\gamma_{\ref{pevelsets}},u_{\ref{pevelsets}})$
s.t.\ $\RLCD^{M^{(1)}}_{\gamma_{\ref{pevelsets}}\sqrt{n},u_{\ref{pevelsets}}}(x)\geq D_0$, $M^{(1)}x=0$}\big\}
\leq 2e^{-n}.
$$
\end{proposition}
\begin{proof}
First, we should carefully define the parameters.
We choose $u:=u_{\ref{p: discrete complete}}(b,\delta,\rho)$.
Next, set $U:=2e^3 C_{\ref{smallball}}^2$,
where $C_{\ref{smallball}}$ is taken from Lemma~\ref{smallball} with parameters
$c_0:=1/2$ and $u/4$, and we assume without loss of generality that $C_{\ref{smallball}}\geq 1$.
Finally, take $\gamma:=\gamma_{\ref{p: discrete complete}}(b,\delta,\rho,U,K)$,
$\gamma':=\widetilde c_{\ref{smallball}} \gamma^2/4\leq 1$.

Let $e^2\leq D\leq D_0\leq e^{\gamma' n}$, and let random matrix $M$ satisfy the assumptions of the proposition.
Let $\Lambda$ be the set defined in Proposition~\ref{p: discrete complete} with $\varepsilon:=1/D_0$.
Set
$$\Event_D:=\big\{
\mbox{$\exists$ $x\in \Incomp(\delta,\rho)\cap S_D(M,\gamma,u)$
s.t.\ $\RLCD^{M^{(1)}}_{\gamma\sqrt{n},u}(x)\geq D_0$, $M^{(1)}x=0$}\big\}.$$
Note that whenever $x$ and $y$ are two vectors in $\R^n$ with
$\RLCD^{M^{(1)}}_{\gamma\sqrt{n},u}(x)\geq D_0$ and $\|x-y\|_\infty\leq \frac{1}{D_0\sqrt{n}}$,
then necessarily $\RLCD^{M^{(1)}}_{\gamma\sqrt{n}/2,u/4}(y)\geq D_0$
(as follows from Lemma~\ref{stable-rlcd}).

Hence, applying Proposition~\ref{p: discrete complete}, we get
\begin{align*}
\Prob(\Event_D)
&\leq e^{-n}+\Prob\big\{\mbox{There is $y\in \Lambda$ with $|M^{(1)} y|\leq \sqrt{n}/D_0$ and
$\RLCD^{M^{(1)}}_{\gamma\sqrt{n}/2,u/4}(y)\geq D_0$}\big\}\\
&\leq e^{-n}+\sharp\Lambda\,\,\sup\limits_{y}\Prob\big\{|M^{(1)} y|\leq \sqrt{n}/D_0\big\}\\
&\leq e^{-n}+(D_0/U)^n\,\,\sup\limits_{y}\Prob\big\{|M^{(1)} y|\leq \sqrt{n}/D_0\big\},
\end{align*}
where the supremum is taken over all vectors $y\in\frac{3}{2}B_2^n\setminus \frac{1}{2}B_2^n$
with $\RLCD^{M^{(1)}}_{\gamma\sqrt{n}/2,u/4}(y)\geq D_0$.

Fix any $y$ satisfying the above conditions.
Set $\widetilde\varepsilon:=2C_{\ref{smallball}}/D_0$ and observe that, by our conditions on $D_0$,
$$
\widetilde\varepsilon\geq C_{\ref{smallball}}\exp(-\widetilde c_{\ref{smallball}} \gamma^2 n/4)
+ C_{\ref{smallball}}/\RLCD^{M^{(1)}}_{\gamma\sqrt{n}/2,u/4}(y).
$$
Applying Lemma~\ref{smallball}, we then obtain
\begin{align*}
\Prob\{|M^{(1)} y|\leq \sqrt{n}/D_0\}&\leq\Prob\{|M^{(1)} y|\leq 2\sqrt{m-1}/D_0\}\\
&\leq\Prob\{|M^{(1)} y|\leq \widetilde\varepsilon\sqrt{m-1}\}\leq  
(C_{\ref{smallball}}\widetilde\varepsilon)^{m-1}.
\end{align*}
Taking the supremum over all admissible $y$, we then get
$$
\Prob(\Event_D)\leq e^{-n}+(D_0/U)^n\,(C_{\ref{smallball}}\widetilde\varepsilon)^{m-1}\leq e^{-n}+
D_0^{n-m+1}U^{-n}\big(2C_{\ref{smallball}}^2\big)^n.
$$
The result follows by the choice of $U$ and the condition on $m$.
\end{proof}

Our proof of Theorem~\ref{mainthm2}, in the case $\Var(A_j)=\Theta(n)$, $j=1,2,\dots,n$,
is a straightforward application of Proposition~\ref{pevelsets} (taking a dyadic sequence of level sets),
together with results of \cite{Liv} on invertibility over compressible vectors.
The fact that in our model some columns may have variances much greater than $n$
adds some complexity to the proof because the relation \eqref{eq: aux 498750983275}
for such columns may hold true only for ``large enough'' $D_0$ leaving a gap in the treatment
of small values of the parameter.
We deal with this issue in the statement below by carefully splitting the event in question
into subevents and invoking Lemma~\ref{l: aux incomp rlcd}
that allows to deterministically bound RLCD in terms of the variance.

\begin{proposition}\label{p: many cases}
Let $b,\delta,\rho\in(0,1)$ and $K\geq 1$ be parameters,
and let $u_{\ref{pevelsets}}$, $\gamma_{\ref{pevelsets}}$ be taken
from Proposition~\ref{pevelsets}.
Then there are $n_{\ref{p: many cases}}(b,\delta,\rho,K)$ and $\gamma'_{\ref{p: many cases}}(b,\delta,\rho,K)$ with the following property.
Let $n\geq n_{\ref{p: many cases}}$, let $n\times n$ matrix $A$ be as in the statement of Theorem~\ref{mainthm2},
and let $j\leq n$ be such that
$$
\Var(A_j)\leq \min\Big(h_{\ref{l: aux incomp rlcd}}^2 e^{-4} n^2,\frac{1}{64}(1-b)\delta\gamma_{\ref{pevelsets}}^2 n^2\Big),
$$
where $h_{\ref{l: aux incomp rlcd}}$ is taken from Lemma~\ref{l: aux incomp rlcd}.
Then
$$
\Prob\big\{\mbox{$\exists$ $x\in \Incomp(\delta,\rho)$ orth.\ to $A_i$, $i\neq j$, with
$\RLCD^{A_j}_{\gamma_{\ref{pevelsets}}\sqrt{n},u_{\ref{pevelsets}}}(x)\leq e^{\gamma'_{\ref{p: many cases}}n}$}\big\}
\leq 2^{-n/2}.
$$
\end{proposition}
\begin{proof}
We will assume that $n$ is large, and that $\gamma'>0$ is a small parameter whose
value can be recovered from the proof below.
Without loss of generality, $j=1$.
Let $A'$ be the submatrix of $A$ composed of all columns $A_i$
satisfying
$$
\Var(A_i)\leq \min\Big(h_{\ref{l: aux incomp rlcd}}^2 e^{-4} n^2,
\frac{1}{64}(1-b)\delta\gamma_{\ref{pevelsets}}^2 n^2\Big).
$$
We note that the number of columns of $A'$ is at least $n-K/\min\big(h_{\ref{l: aux incomp rlcd}}^2 e^{-4},
\frac{1}{64}(1-b)\delta\gamma_{\ref{pevelsets}}^2\big)$.
Further, let $M$ be the transpose of $A'$,
and denote by $W$ the submatrix of $M^{(1)}$ formed by removing rows with variances at least $n^{9/8}$.

The proof of the statement is reduced to estimating probability of the event
$$
\Event':=\big\{\mbox{$\exists$ $x\in \Incomp(\delta,\rho)$ with $M^{(1)}x=0$ and
$\RLCD^{A_1}_{\gamma_{\ref{pevelsets}}\sqrt{n},u_{\ref{pevelsets}}}(x)\leq e^{\gamma'n}$}\big\}.
$$
We can write
\begin{align*}
\Prob(\Event')
\leq
&\sum\limits_{\log_2 n-1\leq \ell\leq \gamma'n\log_2 e}
\Prob\big\{\mbox{$\exists$ $x\in \Incomp(\delta,\rho)\cap S_{2^\ell}(M,\gamma_{\ref{pevelsets}},u_{\ref{pevelsets}})$ with $M^{(1)}x=0$}\big\}
\\
&+
\Prob\big\{\mbox{$\exists$ $x\in \Incomp(\delta,\rho)$ with $M^{(1)}x=0$ and $\RLCD^M_{\gamma_{\ref{pevelsets}}\sqrt{n},
u_{\ref{pevelsets}}}(x)<n$}\big\}.
\end{align*}
The first sum can be estimated directly by applying Proposition~\ref{pevelsets}
with $D_0:=D:=2^\ell$, $\log_2 n-1\leq \ell\leq \gamma' n\log_2 e$
(note that the relation \eqref{eq: aux 498750983275} is fulfilled for such $D$ for all rows of $M$,
and that the proposition can be applied as long as $K/\min\big
(h_{\ref{l: aux incomp rlcd}}^2 e^{-4},\frac{1}{64}(1-b)\delta\gamma_{\ref{pevelsets}}^2\big)\leq
1/\gamma'$).
Further, the condition that $\RLCD^M_{\gamma_{\ref{pevelsets}}\sqrt{n},
u_{\ref{pevelsets}}}(x)<n$ implies that either $\RLCD^W_{\gamma_{\ref{pevelsets}}\sqrt{n},
u_{\ref{pevelsets}}}(x)<n$ or
$\RLCD^W_{\gamma_{\ref{pevelsets}}\sqrt{n},
u_{\ref{pevelsets}}}(x)\geq n$ and $\RLCD^{M^q}_{\gamma_{\ref{pevelsets}}\sqrt{n},
u_{\ref{pevelsets}}}(x)< n$ for some row $M^q$ of $M$.
Hence, we get
\begin{align*}
\Prob(\Event')
\leq 2n\cdot 2e^{-n}
&+\sum_{q}\Prob\big\{\mbox{$\exists$ $x\in \Incomp(\delta,\rho)$ with $Wx=0$ and $\RLCD^W_{\gamma_{\ref{pevelsets}}\sqrt{n},
u_{\ref{pevelsets}}}(x)\geq n$}\\
&\hspace{2cm}\mbox{and $\RLCD^{M^q}_{\gamma_{\ref{pevelsets}}\sqrt{n},
u_{\ref{pevelsets}}}(x)< n$}\big\}\\
&+
\Prob\big\{\mbox{$\exists$ $x\in \Incomp(\delta,\rho)$ with $Wx=0$ and $\RLCD^W_{\gamma_{\ref{pevelsets}}\sqrt{n},
u_{\ref{pevelsets}}}(x)<n$}\big\}.
\end{align*}
To estimate the sum, we apply Lemma~\ref{l: aux incomp rlcd} which, together with our restrictions on the variances,
allows to deterministically bound the RLCD with respect to $M^q$ by $e^2$.
Thus, we get
\begin{align*}
&\Prob\big\{\mbox{$\exists$ $x\in \Incomp(\delta,\rho)$ with $Wx=0$ and $\RLCD^W_{\gamma_{\ref{pevelsets}}\sqrt{n},
u_{\ref{pevelsets}}}(x)\geq n$}\\
&\hspace{2cm}\mbox{and $\RLCD^{M^q}_{\gamma_{\ref{pevelsets}}\sqrt{n},
u_{\ref{pevelsets}}}(x)< n$}\big\}\\
&=\Prob\big\{\mbox{$\exists$ $x\in \Incomp(\delta,\rho)$ with $Wx=0$ and $\RLCD^W_{\gamma_{\ref{pevelsets}}\sqrt{n},
u_{\ref{pevelsets}}}(x)\geq n$}\\
&\hspace{2cm}\mbox{and $e^2\leq \RLCD^{M^q}_{\gamma_{\ref{pevelsets}}\sqrt{n},
u_{\ref{pevelsets}}}(x)< n$}\big\}.
\end{align*}
Splitting the interval $[e^2,n]$ into dyadic subintervals and applying Proposition~\ref{pevelsets} with $D_0:=n$
and for the matrix formed by concatenating $W$ and $M^q$,
we get an upper bound $2e^{-n}\log_2 n$ for the probability.

In order to estimate probability of the event
$$
\big\{\mbox{$\exists$ $x\in \Incomp(\delta,\rho)$ with $Wx=0$ and $\RLCD^W_{\gamma_{\ref{pevelsets}}\sqrt{n},
u_{\ref{pevelsets}}}(x)<n$}\big\},
$$
we apply Lemma~\ref{l: aux incomp rlcd}; this time the definition of $W$ implies that
$\RLCD$ with respect to each row is deterministically bounded from below by $n^{3/8}$, for a sufficiently large $n$.
Again, splitting of the interval $[n^{3/8},n]$ into dyadic subintervals reduces the question to estimating events
of the form
$$
\big\{\mbox{$\exists$ $x\in \Incomp(\delta,\rho)\cap S_D(W,\gamma_{\ref{pevelsets}},u_{\ref{pevelsets}})$ with $Wx=0$}\big\}
$$
for some $D\in [n^{3/8},n]$. Taking $D_0:=D$, one can see that the condition \eqref{eq: aux 498750983275}
is fulfilled for all rows of $W$, and that the difference between the number of columns and rows of $W$
is clearly less than $n/\ln D_0$. Thus, Proposition~\ref{pevelsets} is applicable.

Summarizing, we get $\Prob(\Event')\leq C' ne^{-n}\ln n$
for a universal constant $C'>0$. The result follows for all sufficiently large $n$.
\end{proof}

Now, we are in position to prove Theorem~\ref{mainthm2}.
\begin{proof}[Proof of Theorem~\ref{mainthm2}]
We will assume that $n$ is large.
We start by recording a property of $A$ which follows immediately from Lemma 2.1 (that is, \cite[Lemma~5.3]{Liv}):
For any $j\leq n$, with probability at least $1-e^{-c_1 n}$ any unit vector
orthogonal to $\{A_i,\; i\neq j\}$, 
is $(\delta,\rho)$--incompressible
for some $\delta,\rho\in(0,1)$ depending only on $b,K$ (here, $c_1\in(0,1)$ depends only on $b,K$).
Indeed, let $j\leq n$, let $B$ be the $n\times (n-1)$ matrix formed from $A$
by removing $A_j$, and define $M:=B^\tran.$ Then
$$
\Pr{\exists x\in\Comp(\delta,\rho)\mbox{ orthogonal to }H_j}
\leq \Pr{\inf_{x\in \Comp(\delta,\rho)} |Mx|=0}
\leq e^{-c_1 n},
$$
where in the last passage Lemma 2.1 (that is, \cite[Lemma~5.3]{Liv}) was used. 

Set
$$
r:=\min\Big(h_{\ref{l: aux incomp rlcd}}^2 e^{-4},
\frac{1}{64}(1-b)\delta\gamma_{\ref{pevelsets}}^2\Big),
$$
where $h_{\ref{l: aux incomp rlcd}}$ and $\gamma_{\ref{pevelsets}}$ are defined in respective lemmas with the parameters $b,K,\delta,\rho$.
Pick any index $j\leq n$ such that $\Var(A_j)\leq rn^2$, and let $v$
be a random unit vector orthogonal to $H_j$ and measurable with respect to the sigma-field generated by $H_j$.
Applying Proposition~\ref{p: many cases} together with the above observation, we get
$$
\mbox{$v$ is $(\delta,\rho)$--incompressible and 
$\RLCD^{A_j}_{\gamma_{\ref{pevelsets}}\sqrt{n},
u_{\ref{pevelsets}}}(v) \geq e^{\gamma'_{\ref{p: many cases}} n}$}
$$
with probability at least $1-e^{c_1 n}-2^{-n/2}$.
Application of Lemma~\ref{smallball-1} finishes the proof.
\end{proof}

\begin{remark}
In our proof, the Randomized Least Common Denominator acts like a mediator
in the relationship between anticoncentration properties of matrix-vector products and
cardinalities of corresponding discretizations (nets),
following the ideas developed in \cite{RudVer-square}.
A crucial element of our argument is the fact that RLCD is stable with respect to small perturbations of the vector, which we quantify in Lemma~\ref{stable-rlcd}.

An alternative approach recently considered in \cite{TikhErd}
is based on 
directly estimating the concentration function for ``typical'' points on a multidimensional lattice.
The argument of \cite{TikhErd} uses as an important step certain stability properties
of the L\'evy concentration function and of small ball probability estimates
for linear combinations of Bernoulli random variables.
However, in the general (non-Bernoulli) setting, and with different distributions of entries
of the matrix, obtaining satisfactory
stability properties similar to those in \cite{TikhErd} seems to be a very non-trivial problem, in the situation when the approximation is done by a random vector. We note here that in our net construction the approximating vector is, indeed, random, and depends on the realization of the matrix.

On a technical level, since RLCD is a structural (geometric) property, its stability follows from relatively
simple computations, while the L\'evy concentration function is much more difficult to control;
in particular, the Esseen lemma provides only an upper bound for the concentration function,
hence cannot be relied on when studying its stability.
\end{remark}

\section{Proof of the Theorem \ref{mainthm1}}		\label{s: proof main}

In this section we formally derive Theorem \ref{mainthm1} from Theorem \ref{mainthm2},
using a modification of the ``invertibility via distance'' lemma from \cite{RudVer-square}. 
\begin{lemma}[Invertibility via distance]\label{invdist}
Let $A$ be any $n\times n$ random matrix. Fix a pair of parameters $\delta, \rho\in (0,\frac{1}{2})$, and assume that $n\geq 4/\delta$. Then, for any $\varepsilon>0,$ 
$$\Prob\left\{\inf_{x\in \Incomp(\delta, \rho)} |Ax|\leq \varepsilon \frac{\rho}{\sqrt{n}}\right\}
\leq \frac{4}{\delta n} \inf\limits_{\substack{I\subset[n],\\ \sharp I=n-\lfloor \delta n/2\rfloor}}\sum_{j\in I}
\Prob\{\dist(A_j, H_j)\leq \varepsilon\},$$
where $H_j$ denotes the subspace spanned by all the columns of $A$ except for $A_j.$
\end{lemma}
\begin{proof}
Fix any $I\subset[n]$ with $\sharp I=n-\lfloor \delta n/2\rfloor$,
and consider event
$$
\Event:=\left\{\inf_{x\in \Incomp(\delta, \rho)} |Ax|\leq \varepsilon \frac{\rho}{\sqrt{n}}\right\}.
$$
Fix any realization of the matrix $A$ such that the event holds, i.e.\ there exists a vector
$x\in \Incomp(\delta, \rho)$ with $|Ax|\leq \varepsilon \frac{\rho}{\sqrt{n}}$.
In view of the definition of the set $\Incomp(\delta, \rho)$,
there is a subset $J_x\subset[n]$ of cardinality $\lfloor \delta n\rfloor$ such that
$|x_i|\geq \rho/\sqrt{n}$ for all $i\in J_x$, whence
$$
\dist(A_i,H_i)\leq |x_i|^{-1}\,|Ax|\leq \varepsilon,\quad i\in J_x.
$$
Note that $J_x\cap I$ has cardinality at least $\lfloor \delta n\rfloor-\lfloor \delta n/2\rfloor\geq \delta n/4$.
Thus,
$$
\Event\subset\big\{\sharp\{i\in I:\,\dist(A_i,H_i)\leq \varepsilon\}\geq \delta n/4\big\}
$$
It remains to note that
$$
\Prob\big\{\sharp\{i\in I:\,\dist(A_i,H_i)\leq \varepsilon\}\geq \delta n/4\big\}
\leq \frac{4}{\delta n}\E\,\sharp\{i\in I:\,\dist(A_i,H_i)\leq \varepsilon\}.
$$
\end{proof}

\textbf{Proof of Theorem \ref{mainthm1}.} The theorem follows from Lemma \ref{comp-final} (that is, Lemma 5.3 from \cite{Liv}), Lemma \ref{invdist} and Theorem \ref{mainthm2},
by taking $I_0:=\{i\in[n]:\;\E|A_i|^2\leq rn^2\}$ and noting that, in view of the assumption $\E\|A\|_\HS^2\leq Kn^2$, 
we have $\sharp I_0=n-K/r\geq n-\lfloor \delta n/2\rfloor$ for all sufficiently large $n$, so that for all large enough $n$
$$
\Prob\left\{\inf_{x\in \Incomp(\delta, \rho)} |Ax|\leq \varepsilon \frac{\rho}{\sqrt{n}}\right\}
\leq \frac{4}{\delta n} \sum_{j\in I_0}
\Prob\{\dist(A_j, H_j)\leq \varepsilon\}.
$$

$\square$

\end{document}